\documentclass[12pt]{article}

\usepackage[top=50mm, bottom=50mm, left=50mm, right=50mm]{geometry}
\usepackage{lineno}
\usepackage{amssymb}
\usepackage{amsmath}
\usepackage{amsthm}
\usepackage{bbold}
\usepackage{epsfig}
\usepackage{faktor}
\usepackage{graphicx}
\usepackage{graphics}
\usepackage{pifont}
\usepackage{float}
\usepackage{subfigure}
\usepackage{multirow}
\usepackage{color}
\usepackage{lineno}
\usepackage{fullpage}
\usepackage[normalem]{ulem} 
\usepackage{makeidx}
\usepackage{xspace}
\usepackage{wrapfig}
\usepackage{todonotes}
\usepackage{hyperref}
\usepackage{cleveref}
\usepackage{pst-node}
\usepackage{tikz-cd} 
\usepackage{mathtools}
\makeindex

\newtheorem{Definition}{Definition}[section]

\newtheorem{Theorem}[Definition]{Theorem}
\newtheorem{Lemma}[Definition]{Lemma}

\newtheorem{Remark}[Definition]{Remark}
\newtheorem*{theorem*}{Theorem}
\newtheorem{Question}[Definition]{Question}
\newtheorem{thm}{Theorem}

\definecolor{darkred}{rgb}{1, 0.1, 0.3}
\definecolor{darkblue}{rgb}{0.1, 0.1, 1}
\definecolor{darkgreen}{rgb}{0,0.6,0.5}

\newcommand {\mm}[1] {\ifmmode{#1}\else{\mbox{\(#1\)}}\fi}

\newcommand{\C}{\mathbb{C}} 
\newcommand{\cp}{\mathbb{C}\mathrm{P}} 
\newcommand{\rp}{\mathbb{R}\mathrm{P}} 
\newcommand{\R}{\mathbb{R}} 
\newcommand{\Z}{\mathbb{Z}} 
 
\newcommand{\HH}{\mathcal{H}}
\newcommand{\MM}{\mathcal{M}}

\newcommand{\VV}{\mathcal{V}}
\newcommand{\dd}{\mathrm{d}}

\newcommand{\diff}{\frac{\mathrm{d}}{\mathrm{d}t}\Big|_{t=0}}

\begin{document}

\title{Hofer--Zehnder capacity of disc tangent bundles of projective spaces}
 
\author{{Johanna Bimmermann}}

\maketitle

\begin{abstract}
\noindent
We compute the Hofer--Zehnder capacity of disc tangent bundles of the complex and real projective spaces of any dimension. The disc bundle is taken with respect to the Fubini-Study resp.\ round metric, but we can obtain explicit bounds for any other metric. In the case of the complex projective space we also compute the Hofer--Zehnder capacity for the magnetically twisted case, where the twist is proportional to the Fubini-Study form. For arbitrary twists we can still give explicit upper bounds.
\end{abstract}

\section{Introduction}
\noindent
The Hofer--Zehnder capacity \cite[Chpt.\ 3]{HZ94} belongs to the class of symplectic capacities, i.e.\ numerical symplectic invariants measuring the symplectic size of a symplectic manifold. For a compact symplectic manifold $(M,\omega)$ possibly with boundary $\partial M$ the Hofer--Zehnder capacity is defined as follows
$$
c_{HZ}(M,\omega):=\sup\left\lbrace \mathrm{max}(H)\ \vert\ H: M\to \R\ \  \text{smooth, admissible}\right\rbrace,
$$
where admissible means:
\begin{itemize}
    \item $0\leq H$ and there exists an open set $U\subset M\setminus\partial M$ such that $H\vert_U\equiv 0$,
    \item there exists a compact set $K\subset M\setminus\partial M$ such that $H\vert_{M\setminus K}\equiv\max(H)$,
    \item all non-constant periodic solutions $\gamma:\R \to M$ of $\dot \gamma=X_H$ have period $T>1$.
\end{itemize}
 Here, $X_H$ denotes the Hamiltonian vector field defined imposing the relation $\dd H=-\iota_{X_H}\omega$. As all symplectic capacities, also the Hofer--Zehnder capacity is fairly hard to compute and unknown in many cases. A classic class of examples for symplectic manifolds are (co-) tangent bundles. Recall that the cotangent bundle of any smooth manifold $N$ admits a canonical symplectic structure $\dd\lambda$, where $\lambda\in \Omega^1(T^*N,\R)$ is pointwise defined as follows
$$
\lambda_{(x,p)}: T_{(x,p)}TN\to \R;\ \xi \mapsto p(\dd\pi _{(x,p)}(\xi))\ \ \ \forall\ (x,p)\in T^*N,
$$
where $\pi: TN\to N$ denotes the footpoint projection.
Fixing a Riemannian metric $g$ on $N$ we can pull back $\lambda$ to $TN$ via the metric isomorphism. By abuse of notation this pullback will also be denoted by $\lambda$. We will now study the disc subbundles with respect to the already chosen metric $g$,
$$
D_\rho N:=\lbrace (x,v)\in TN\ \vert \ g_x(v,v)< \rho^2\rbrace
$$
and investigate the following question.
\begin{Question}\label{q1}
What is the value of $c_{HZ}(D_\rho N,\dd\lambda)$?
\end{Question}
\noindent
To the authors knowledge this question is still wildly open and most of the results in this direction concern finiteness. For example the Hofer--Zehnder capacity is finite, if the base manifold carries a circle action with non-contractible orbits \cite{Irie14}, if the Hurewicz map $\pi_2(N)\to H_2(N,\Z)$ is non zero \cite{AFO17} or if $N$ is rationally inessential \cite{FP17}. All these results rely on the observation by Irie \cite[Cor. 3.5]{Irie14} that vanishing of symplectic homology implies finiteness of the Hofer--Zehnder capacity. Finiteness can also be shown by finding a Lagrangian embedding of $N$ into a manifold with finite Hofer--Zehnder capacity, as the Lagrangian neighborhood theorem assures that then also a small disc bundle $(D_\varepsilon N,\dd\lambda)\hookrightarrow (M,\omega)$ embeds symplectically. This is also the easiest way to see that for the two examples we will treat $c_{HZ}(D_1\cp^n,\dd\lambda)$ and $c_{HZ}(D_1\rp^n,\dd\lambda)$ are finite. Indeed we have Lagrangian embeddings $\cp^n\hookrightarrow (\cp^n\times\cp^n,\omega_{FS}\ominus\omega_{FS})$ and $\rp^n\hookrightarrow (\cp^n,\omega_{FS})$.\\ 
\ \\
However explicit values of the Hofer--Zehnder capacity of disc tangent bundles are only given for certain flat tori \cite[Ch.4.4, Prop. 4]{HZ94} and some convex subsets of $\R^n$ \cite{AKO14}, using that in these cases the tangent bundles can be trivialized and then identified with subsets of $\R^{2n}$.\\
\ \\
There is a variation of the Hofer--Zehnder capacity for which the analog of Question \ref{q1} is partially answered. For a general compact symplectic manifold $(M,\omega)$ possibly with boundary $\partial M$ consider a closed submanifold $\Sigma\subset M$ not intersecting the boundary, i.e. $\Sigma\subset M\setminus \partial M$. Then the Hofer--Zehnder capacity relative to $\Sigma$ is defined as
$$
c_{HZ}(M,\Sigma,\omega):=\sup\left\lbrace \mathrm{max}(H)\ \vert\ H: M\to \R\ \  \text{smooth, admissible and vanishing on }\Sigma\right\rbrace.
$$
For any homotopy class $\nu\in \pi_1(M)$ we say $H: M\to \R$ is $\nu$-admissible if $H$ is smooth, $0\leq H$, vanishes on an open set $U\subset M\setminus\partial M$, constantly attains its maximum on the complement of a compact set $K\subset M\setminus\partial M$ and all non-constant periodic solutions $\gamma:\R \to M$ of $\dot \gamma=X_H$ with $[\gamma]=\nu$ have period $T>1$. One defines the $\pi_1$-sensitive Hofer--Zehnder capacity $c^\nu_{HZ}(M,\Sigma,\omega)$, replacing admissible with $\nu$-admissible. In general we have the following chain of inequalities
$$
c_{HZ}(M,\Sigma,\omega)\leq c_{HZ}^\nu(M,\Sigma,\omega)\leq c_{HZ}^\nu(M,\omega)
$$
and also
$$
c_{HZ}(M,\Sigma,\omega)\leq c_{HZ}(M,\omega)\leq c_{HZ}^\nu(M,\omega),
$$
simply because the class of admissible Hamiltonians is less and less constrained.
An important result by Weber \cite{Wbr06} shows that for closed $N$ and a non-zero class $\nu\in\pi_1(N)$, the Hofer-Zehnder capacity relative to the zero section is given by the length $l_\nu$ of the shortest closed geodesic in the class $\nu\in\pi_1(M)$, i.e.
\begin{equation}
    c_{HZ}^\nu(D_1N,N,\dd\lambda)=l_\nu.
\end{equation}
For non-aspherical, homogeneous spaces and positive curvature metrics on the 2-sphere the result was extended to the class of contractible loops ($\nu=0$) by Benedetti and Kang \cite[Cor.\ 2.8]{BK22}. Observe that restricting to the case $\rho=1$ is fine as scaling the fibers yields a symplectomorphism identifying $(D_\rho N,\dd\lambda)\cong (D_1 N,\rho\cdot \dd\lambda)$.\\
\ \\
In this article we derive the precise value of the Hofer--Zehnder capacity of $(D_1\cp^n,\dd\lambda)$ and $(D_1\rp^n,\dd\lambda)$ with respect to the Fubini-Study metric respectively the round metric. 
\begin{thm}\label{thma}
    Equip $\cp^n$ with the Fubini-Study metric and $\rp^n$ with the round metric. Denote by $l$ the length of the prime geodesics, then
    $$
    c_{HZ}(D_1 \cp^n, \dd\lambda)=l,
    $$
    while
    $$
    c_{HZ}(D_1 \rp^n, \dd\lambda)=2l.
    $$
\end{thm}
\noindent
Observe that in the case of $\cp^n$ it coincides with the relative capacity, while in the case of $\rp^n$ there is a factor of two. This factor of two is an example of the phenomenon of Lagrangian barriers first described by Biran \cite{BR01} and recently demonstrated impressively by Brendel and Schlenk \cite{BS22}.

\begin{Remark}
    We already made two choices here. First the metric and second the neighborhood of the zero section. Both are in some sense related. Any two metrics $g$ and $h$ are related by a symmetric positive definite bundle map $A: TN\to TN$ such that for all $(x,v)\in TN$ we have
    $$
    h_x(v,v)= g_x(A_x v,v).
    $$
    Further denote $\lambda_g$ resp.\ $\lambda_h$ the pullback of the canonical 1-form with respect to $g$ resp.\ $h$. Then these two 1-forms can also be identified via the bundle map $A$, i.e. $A^*\lambda_g=\lambda_h$. In particular the  unit disc subbundle $D_1N$ is identified with some ellipsoid subbundle 
    $$
    E_AN:=\lbrace (x,v)\in TN\ \vert \ g_x(v, A_x^{-1}v)\leq 1\rbrace.
    $$
    Therefore $E_AN\subset D_aN$ where $a:N\to\R_{\geq 0}$ is the maximal eigenvalue of $A_x^{-1}$. In particular if $N$ is compact we can obtain the following upper bound for the Hofer--Zehnder capacity
    $$
    c_{HZ}(D_1N,\dd\lambda_h)\leq\max_{x\in N} a(x)\cdot c_{HZ}(D_1N,\dd\lambda_g).
    $$
\end{Remark}
\noindent
We will also look at the magnetically twisted tangent bundle $(D_\rho \cp^n,\dd\lambda-s\pi^*\omega_{FS})$ for arbitrary constants $\rho>0, s\in \R$. We again have a scaling property, in this case
$$
(D_\rho \cp^n,\dd\lambda-s\pi^*\omega_{FS})\cong \left(D_1 \cp^n,\rho\left (\dd\lambda-\frac{s}{\rho}\pi^*\omega_{FS}\right)\right).
$$
It is therefore enough to restrict to the case $\rho=1$ and $s\in\R$ arbitrary.
\begin{thm}\label{thmb} Denote $\omega_{FS}$ the Fubini-Study form and $l$ the length of the prime geodesics, then
    $$
     c_{HZ}(D_1 \cp^n,\dd\lambda-s\pi^*\omega_{FS})=l\left(\sqrt{s^2+ 1}-\vert s\vert \right).
    $$
\end{thm}
\begin{Remark}
    We can also say something about the Hofer--Zehnder capacity of non-constant magnetic systems on $\cp^n$. Let $\nu\in\Omega^2(\cp^n)$ be an arbitrary closed 2-form. As $H_{dR}^2(\cp^n)$ is generated by $[\omega_{FS}]$ we know that $[\nu]=[s\omega_{FS}]$ in $H_{dR}^2(\cp^n)$ for some $s\in\R$. Thus we find an exact form $\dd\theta$ such that $\nu+\dd\theta=s\omega_{FS}$. But this means $(T^*\cp^n,\dd\lambda-\pi^*\nu)$ is symplectomorphic to $(T^*\cp^n,\dd\lambda-s\pi^*\omega_{FS})$
via the map that shifts the zero-section
$$
(x,p)\mapsto(x,p+\theta(x)).
$$
Clearly this map does not map disc bundles to disc bundles but we can still get some inclusions depending on $\theta\in\Omega^1(\cp^n)$. Denote
$$
\theta_{\max}= \max_x\vert \theta(x)\vert,
$$
then 
$$
 (D_{1}\cp^n,\dd\lambda-\pi^*\nu)\hookrightarrow(D_{1+\theta_{\max}}\cp^n,\dd\lambda-s\pi^*\omega_{FS}).
$$
\noindent
There is a freedom of choosing the primitive $\theta$ that we can use to optimize the bound. Set
$$
\rho=1+\inf_\theta\max_x\vert \theta(x)\vert,
$$
where the infimum is taken over all 1-forms $\theta\in \Omega^1(M)$ satisfying $\nu+\dd\theta=s\omega_{FS}$. Then
    $$
    c_{HZ}(D_1\cp^n, \omega_\nu)\leq l\left(\sqrt{s^2+\rho^2}-\vert s\vert\right).
    $$
\end{Remark}
\noindent
The proofs of Theorem \ref{thma} and Theorem \ref{thmb} evolve in two steps: finding a lower bound and finding an upper bound that coincides with the lower bound. For the lower we will explicitly construct an admissible Hamiltonian using geodesic billiards and the magnetic geodesic flow respectively. For the upper bound we will need more abstract theory, the theory of pseudo-holomorphic curves. There are not many cases where this theory becomes explicit, but we will construct the following symplectomorphisms, realizing compactifications of the disc bundles, that will put us in the perfect set up to study (pseudo-)holomorphic curves. In this sense the following two theorems are maybe the key contribution of this paper. \\
\begin{thm}\label{thmc} Denote $\omega_{FS}\in \Omega_2(\cp^n)$ the Fubini-Study form normalized such that the generator of $H_2(\cp^n,\Z)$ has area $4\pi$. Then there is a symplectomorphism, which is equivariant with respect to the Hamiltonian $\mathrm{SU}(n+1)$ actions, 
$$
F: (D_1 \cp^n,\dd\lambda-s\pi^*\omega_{FS})\to (\cp^n\times \cp^n\setminus \bar\Delta, R_1\omega_{FS}\ominus R_2\omega_{FS}),
$$
where $\bar\Delta\subset \cp^n\times \cp^n$ denotes the anti-diagonal divisor
$$
\bar\Delta=\lbrace (p,q)\in\cp^n\times\cp^n\ \vert \ \mathrm{dist}(p,q) \ \text{maximal}\rbrace
$$
and 
$$
R_1=\frac{1}{2}\left (\sqrt{s^2+1}+s\right),\ \ R_2=\frac{1}{2}\left (\sqrt{s^2+1}-s\right).
$$
\end{thm}
\noindent
Observe that the case $s=0$ is included in the Theorem and recovers the untwisted case. For $\rp^n$ we have a similar theorem.
\begin{thm}\label{thmd}Denote $\omega_{FS}\in \Omega_2(\cp^n)$ the Fubini-Study form normalized such that the generator of $H_2(\cp^n,\Z)$ has area $4\pi$.
    Then there is an $SO(n+1)$-equivariant symplectomorphism
    $$
    F: (D_{1}\R\mathrm{P}^n,\dd\lambda)\to (\cp^n\setminus Q^{n-1},\omega_{FS}),
    $$
    where $Q^{n-1}\subset\cp^n$ denotes the quadric 
    $$
    Q^{n-1}:=\lbrace [z_0:\ldots:z_n]\ \vert \ z_0^2+\ldots+z_n^2=0\rbrace\subset\cp^n. 
    $$
\end{thm}
\noindent
Note that the 2-dimensional version of this theorem was already proven by Adaloglou in \cite{Ad22}, but using a different approach.\\
\ \\
\noindent
\textbf{Outline.} In section \ref{sec1} we construct the symplectic compactfications of $(D_1\cp^n,\dd\lambda)$ and $(D_1\rp^n,\dd\lambda)$, i.e. we prove Theorem \ref{thmc} in the case $s=0$ and Theorem \ref{thmd}. Section \ref{sec2} contains a recap on pseudoholomorphic and uses them to obtain an upper bound on the Hofer--Zehnder capacity. Section \ref{sec3} is devoted to constructing admissible Hamiltonians to get a lower bound. We will see that upper and lower bound agree and conclude that Theorem \ref{thma} holds. The last section (section \ref{sec4}) contains the magnetic version, i.e. the proofs of Theorem \ref{thmb} and Theorem \ref{thmd}.  \\
\ \\
\noindent
\textbf{Acknowledgments.} I want to thank my advisor Gabriele Benedetti for his support and suggestions on how to deal with numerous difficulties. Further I also want to thank Peter Albers and Kai Zehmisch for helpful discussions regarding this work. The author further acknowledges funding by the Deutsche Forschungsgemeinschaft (DFG, German Research Foundation) – 281869850 (RTG 2229) and 281071066 (TRR 191).

\section{Compactifications}\label{sec1}
Let $N$ denote either $\cp^n$ or $\rp^n$. We equip $\cp^n$ with the Fubini-Study metric and $\rp^n$ with the round metric, both denoted by $g$. Recall that the kinetic Hamiltonian
$$
E(x,v):=\frac{1}{2}g_x(v,v)\left (=\frac{1}{2}\vert v\vert_x^2\right)
$$
generates geodesic flow and since we chose Zoll metrics it is totally periodic of period $l/ \vert v\vert$, where $l$ denotes the length of the prime geodesics. It follows that the Hamiltonian
$$
H:=l \sqrt{2E}: TN\setminus 0_{TN}\to \R;\ (x,v)\mapsto l \vert v\vert_x
$$
generates a free circle action. One can now perform a Lerman cut (see \cite{ler}) at $H=l$ to obtain symplectic compactifications $$ 
\overline{(D_1\cp^n,\dd\lambda)} \ \text{resp.}\ \overline{(D_{1}\rp^n,\dd\lambda)}.
$$
These compactifications are per construction closed symplectic, but their symplectic geometry is not very clear from the Lerman cut construction. This section is concerned with identifying these 'abstract' compactifications with the well studied symplectic manifolds 
$$\left(\cp^n\times \cp^n,\frac{1}{2}(\omega_{FS}\ominus\omega_{FS})\right)\ \text{resp.}\ \left(\cp^n,\omega_{FS}\right),
$$
which puts us in the perfect set up to study (pseudo-) holomorphic spheres. To find the identifying symplectomorphism, we will make use of the fact that the action of the isometry groups $\mathrm{SU}(n+1)$ resp. $\mathrm{SO}(n+1)$ of $\cp^n$ resp. $\rp^n$ lift to Hamiltonian actions on the tangent bundle. We will see that indeed also the spaces $\left(\cp^n\times \cp^n,(\omega_{FS}\ominus\omega_{FS})\right)\ \text{resp.}\ \left(\cp^n,\omega_{FS}\right)$ admit Hamiltonian $\mathrm{SU}(n+1)$- resp. $\mathrm{SO}(n+1)$-actions. The map we construct will be an equivariant smooth bijection and we will use the following lemma to deduce it is a symplectomorphism.
\begin{Lemma}\label{lem1}
Let $G$ be a Lie group and $(M_i,\omega_i)$, $i=1,2$ symplectic manifolds. Assume we have a Hamiltonian G-action on $(M_i, \omega_i)$ with moment maps $\mu_i: M_i\to\mathfrak{g}^*$ for $i=1,2$. If $\varphi: M_1\to M_2$ is an equivariant smooth bijection such that 
\[
\begin{tikzcd}
M_1 \arrow[dr,"\mu_1"] \arrow[rr,"\varphi"] && M_2\arrow[dl,"\mu_2"]  \\
 & \mathfrak{g}^* 
\end{tikzcd}
\]
commutes and the distribution $\mathcal{D}\subset TM_1$ tangent to the $G$-orbits admits a complement $\Upsilon$ that is isotropic for both symplectic forms $\omega_1$ and $\varphi^*\omega_2$, then $\varphi$ is actually a symplectomorphism i.e.\ $\varphi^*\omega_2=\omega_1$. 
\end{Lemma}
\begin{proof}
We need to show that 
$$
\varphi^*\omega_2(v,w)=\omega_1(v,w)\ \ \forall v,w\in TM_1.
$$
By non-degeneracy of $\omega_1$ and $\omega_2$ it will then follow, that also $\dd\varphi$ must be non-degenerate hence a diffeomorphism.
We can decompose the tangent bundle $TM_1=\mathcal{D}\oplus\Upsilon$. The distribution tangent to the $G$-orbits is given by $\mathcal{D}_x=\lbrace (a_1^\#)_x:=\diff e^{ta}(x)\ \vert\  a\in\mathfrak{g}\rbrace$. In particular $\dd\varphi a_1^\#=a_2^\#$ by equivariance, where the subscript indicates the manifold the vector field induced by $a\in\mathfrak{g}$ lives on. So if either $v=a_1^\#\in\mathcal{D}$ or $w=a_1^\#\in\mathcal{D}$ we get the equality as follows
$$
\omega_1(a^\#_1,\cdot)=\dd(\mu_1,a)= \dd(\mu_2\circ\varphi,a)=\omega_2(a^\#_2,\dd\varphi\ \cdot)=\omega_2(\dd\varphi\ a_1^\#, \dd\varphi\ \cdot)=\varphi^*\omega_2(a^\#_1,\cdot).
$$
The second equality holds due to the moment map triangle and the fourth due to equivariance of $\varphi$. If $v,w\in \Upsilon$ we get
$$
\omega_1(v,w)=0=\omega_2(\dd\varphi( v),\dd\varphi( w))
$$
as $\Upsilon$ is isotropic with respect to both symplectic forms by assumption.
\end{proof}
\subsection{Complex projective space}
In order to give an accurate description of the Hamiltonian group actions, including formulas for the moment maps, we will consider $\cp^n$ as (co-)adjoint orbit in $\mathfrak{su}(n+1)$ and $\rp^n\subset\cp^n$ as a suborbit. Observe that for compact Lie groups such as $\mathrm{SU}(n)$ we can use the Killing form to identify $\mathfrak{g}^*\cong\mathfrak{g}$, which we will do without further notice from now on. We follow \cite{Ef04} for the presentation of $\cp^n$ as adjoint orbit. The complex projective space can be identified with the homogeneous space
$$
\cp^n=\faktor{\mathrm{SU}(n+1)}{\mathrm{S}(\mathrm{U}(n)\times\mathrm{U}(1))},
$$
where $\mathrm{S}(\mathrm{U}(n)\times\mathrm{U}(1))$ denotes the matrices in $\mathrm{U}(n)\times\mathrm{U}(1)$ of determinant one. Fix
$$
Z:=\frac{1}{n+1}\begin{pmatrix}
-i_n & \quad \\
\quad & ni 
\end{pmatrix}\in \mathfrak{su}(n+1).
$$
The stabilizer of $Z$ under the adjoint action is precisely $\mathrm{S}(\mathrm{U}(n)\times\mathrm{U}(1))$, thus $\cp^n$ can be realized as the adjoint orbit $O_Z\subset \mathfrak{su}(n+1)$.
The stabilizer group $\mathrm{S}(\mathrm{U}(n)\times\mathrm{U}(1))$ is isomorphic to $\mathrm{U}(n)$ via the map
$$
\mathrm{U}(n)\to\mathrm{S}(\mathrm{U}(n)\times\mathrm{U}(1));\ A\mapsto \begin{pmatrix}
A & \quad \\
\quad & \det A^{-1}
\end{pmatrix}.
$$
On Lie algebra level the inclusion looks like
$$
\mathfrak{u}(n)\hookrightarrow \mathfrak{su}(n+1);\ B\mapsto \begin{pmatrix}
B & \quad \\
\quad & -\mathrm{tr} B
\end{pmatrix}.
$$
Thus the Lie-algebra splits as
$$
\mathfrak{su}(n+1)=\mathfrak{p}\oplus\mathfrak{u}(n),
$$
where 
$$
\mathfrak{p}=\left\lbrace \begin{pmatrix}
0_n & -\bar z \\
z^t & 0
\end{pmatrix}\Big\vert\  z\in\C^n\right\rbrace
$$
is the subspace corresponding to the tangent space of $\cp^n$ at $Z$. Observe that for any 
$$
v=\begin{pmatrix}
0_n & -\bar z \\
z^t & 0
\end{pmatrix}\in T_Z\cp^n=\mathfrak{p}
$$
we have that
$$
j_Z v := [Z,v] = \frac{1}{n+1}\left(\begin{pmatrix}
-i_n & \quad \\
\quad & ni 
\end{pmatrix}\begin{pmatrix}
0_n & -\bar z \\
z^t & 0
\end{pmatrix}-\begin{pmatrix}
0_n & -\bar z \\
z^t & 0
\end{pmatrix}\begin{pmatrix}
-i_n & \quad \\
\quad & ni 
\end{pmatrix}\right)=\begin{pmatrix}
0_n & -(\overline{iz}) \\
(iz)^t & 0
\end{pmatrix}
$$
recovers the standard complex structure on $\cp^n$. Further the negative of the Killing form 
$$
(X,Y):=-2\mathrm{tr}(X\cdot Y),\ \ \ \forall\ X,Y\in \mathfrak{su}(n+1)
$$
yields an $\mathrm{Ad}$-invariant scalar product on $\mathfrak{su}(n+1)$. Restricted to $\mathfrak{p}$ it is compatible with $j_Z$ and equivariantly extended it yields the standard invariant Kähler structure of $\cp^n$, i.e. $g(j\cdot,\cdot)=\omega_{FS}(\cdot,\cdot)$. 
\begin{Lemma}
    The induced resp.\ diagonal group action of $\mathrm{SU}(n+1)$ on $(T\cp^n,\dd\lambda)$ resp.\ $(\cp^n\times\cp^n, \omega_{FS}\ominus \omega_{FS})$ is Hamiltonian. The moment maps are given by 
    $$
    \mu_1:T\cp^n\to\mathfrak{su}(n+1);\ (x,v)\mapsto [x,v]
    $$
    resp.\
    $$
    \mu_2: \cp^n\times\cp^n \to \mathfrak{su}(n+1);\ (a,b)\mapsto a-b.
    $$
\end{Lemma}
\begin{proof} 
    For the first formula we compute for an arbitrary element $a\in \mathfrak{su}(n+1)$
    $$
        \dd (([x,v],a))=\dd ((v,[x,a]))=\dd (g(v,\dd\pi(a^\#)))=\dd(\lambda(a^\#))=\iota_{a^\#}\dd\lambda
    $$
    as $\lambda$ is invariant under the flow of $a^{\#}$ and as a consequence $\mathcal{L}_{a^{\#}}\lambda=0$.\\
    \ \\
    The formula for $\mu_2$ follows immediately from the fact that the inclusion as adjoint orbit is a moment map for the group action on an adjoint orbit \cite[Ch.\ 1]{Kr04}.
\end{proof}
\noindent
We can now prove the following theorem using the Lemmas above.
\begin{Theorem}\label{thm1}
There is a symplectomorphism, which is equivariant with respect to the Hamiltonian $\mathrm{SU}(n+1)$ actions, 
$$
F: (D_2 \cp^n,\dd\lambda)\to (\cp^n\times \cp^n\setminus \bar\Delta, \omega_{FS}\ominus\omega_{FS}),
$$
where $\bar\Delta\subset \cp^n\times \cp^n$ denotes the anti-diagonal divisor
$$
\bar\Delta=\lbrace (a,b)\in\cp^n\times\cp^n\ \vert \ \mathrm{dist}(a,b) \ \text{maximal}\rbrace.
$$
\end{Theorem}
\begin{proof}
We make the ansatz
$$ F: D_2 \cp^n\to \cp^n\times \cp^n\setminus\bar\Delta;\ \ \ (x,v)\mapsto \left(\exp_x(-c(r)j_xv),\exp_x(c(r)j_xv)\right),
$$
for a function $c$ depending on $r:=\vert v\vert$.
The map $F$ is depicted in Figure \ref{f1} for the two dimensional case of $\cp^1$.
\begin{figure}
	\centering
 \includegraphics[width=0.4\textwidth]{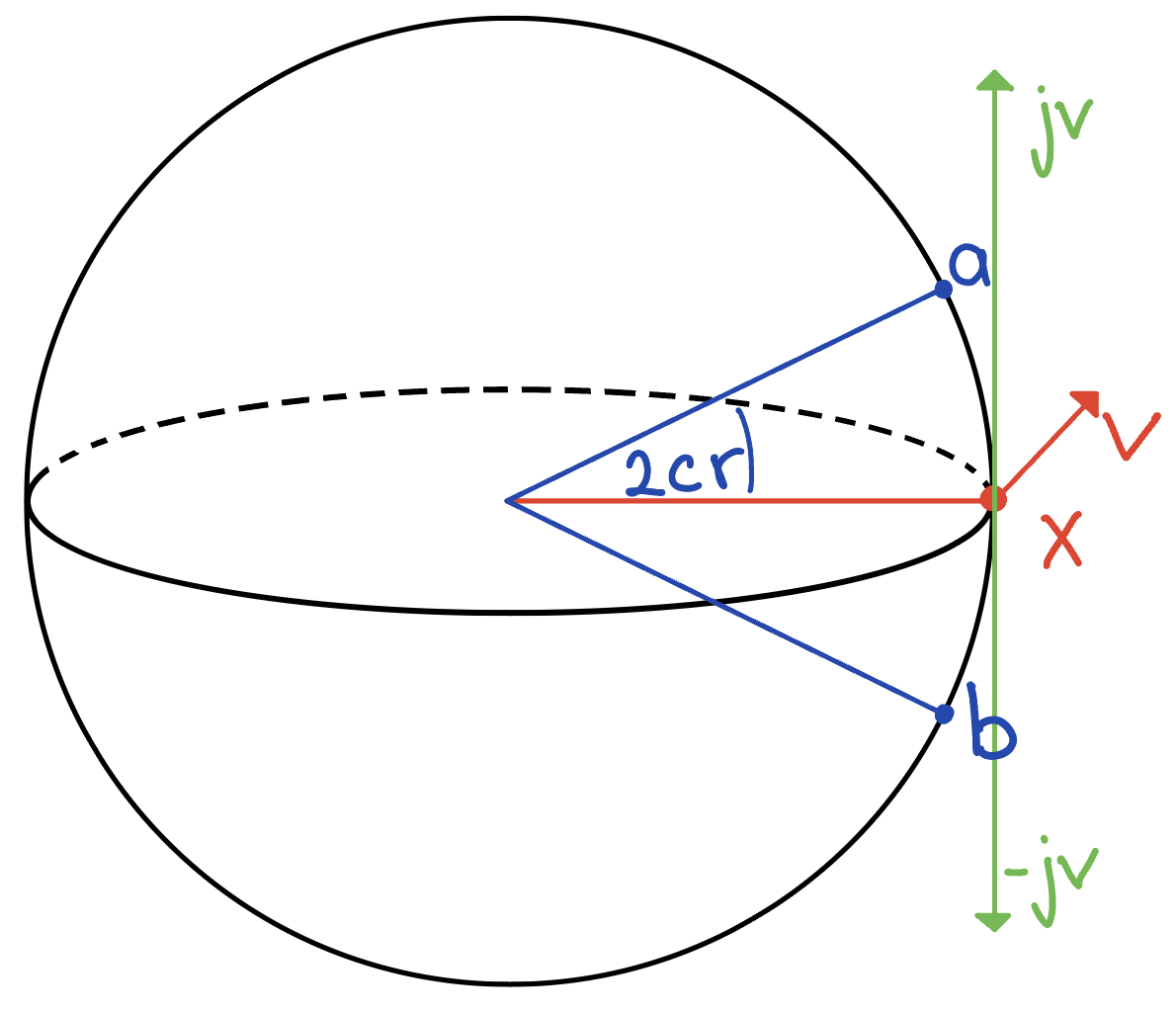}
	\caption{\textit{The Lie algebras $\mathfrak{su}(2)$ and $\R^3$ with cross product as Lie bracket are isomorphic. This isomorphism identifies the adjoint orbit $\cp^1$ with the round 2-sphere. Now take a point $(x,v)\in TS^2$ (in red). If we fix the radius of the sphere to be one the complex structure is given by $jv=x\times v$ (in green). This is actually the same as the moment map, i.e. $\mu_1(x,v)=x\times v=jv$. We follow the geodesic (great circle) through $x$ in direction $jv$ resp.\ $-jv$ until the points $a$ and $b$ satisfy $\mu_2(a,b)=a-b=jv$. The angle between $a$ and $b$ uniquely determines the function $c$. That the higher dimensional case works the same is no surprise as one can use the symmetry group $\mathrm{SU}(n+1)$ to reduce everything to the two dimensional case.  }}
    \label{f1}
\end{figure}
For the proof we want to use Lemma \ref{lem1}, thus we need to show that the moment map triangle commutes. We will determine $c$ imposing the relation of the moment maps 
$$
\mu_1(x,v)=\mu_2(F(x,v)).
$$
Per construction the moment maps and the map $F$ are equivariant. Hence, it is enough to check the moment map triangle at some point
$$
x=Z=\frac{1}{n+1}\begin{pmatrix}
-i_{n-1} & \quad & \quad\\
\quad & -i & 0\\
\quad & 0 & ni
\end{pmatrix},\ \ v=\frac{r}{2}\begin{pmatrix}
0_{n-1} & \quad & \quad\\
\quad & 0 & -i\\
\quad & -i & 0
\end{pmatrix} \in T_Z\cp^n,\ r=\vert v\vert>0,
$$
as the group action is transitive on the sphere subbundle of $T\cp^n$.
We compute 
$$
\mu_1(x,v)=[x,v]=\frac{r}{2}\begin{pmatrix}
0_n & \quad & \quad\\
\quad & 0 & 1\\
\quad & -1 & 0
\end{pmatrix}.
$$
The geodesic through $x$ in direction $j_x v$ is given by
$$
\gamma_{(x,j_xv)}(t)=\mathrm{Ad}_{\exp(-tv)}x=\frac{1}{n+1}\begin{pmatrix}
-i_{n-1} & \quad & \quad\\
\quad & -i(\cos(rt)^2-n\sin(rt)^2) & (n+1)\sin(rt)\cos(rt)\\
\quad & -(n+1)\sin(rt)\cos(rt) & i(n\cos(rt)^2-\sin(rt)^2)
\end{pmatrix}.
$$
It parametrizes an affine circle in $\mathfrak{su}(n+1)$ since
$$
V(t):=\frac{1}{2}\left(\gamma_{(x,j_xv)}(t)-\gamma_{(x,j_x v)}(t+\frac{\pi}{2r})\right)=\frac{1}{2}\begin{pmatrix}
0_{n-1} & \quad & \quad\\
\quad & -i\cos(2rt) & \sin(2rt)\\
\quad & -\sin(2rt) & i\cos(2rt)
\end{pmatrix}
$$
satisfies $\vert V\vert^2=-2\mathrm{tr}(V^2)=1$.
The circle is centered at
$$
y:=\frac{1}{2}\left(\gamma_{(x,j_x v)}(t)+\gamma_{(x,j_x v)}(t+\frac{\pi}{2r})\right)=\frac{1}{n+1}\begin{pmatrix}
-i_{n-1} & 0 & 0\\
0 & \frac{n-1}{2}i & 0\\
0 & 0 & \frac{n-1}{2}i
\end{pmatrix}.
$$
Observe that $\gamma_{(x,j_xv)}(t)=V(t)+y$.
\begin{Remark}
    The equality $\vert V\vert=1$ shows that the Fubini-Study form is normalized such that prime geodesics have length $l=2\pi$.
\end{Remark}
\noindent
Now the second moment map is given by
$$
\mu_2(F(x,v))=\gamma_{(x,j_xv)}(-c)-\gamma_{(x,j_xv)}(c)=V(-c)-V(c).
$$
Comparing the coefficients of the two matrices yields the following equation for $c$:
$$
    2\sin(2cr)=r.
$$
The moment map triangle commutes if and only if this equality holds. We immediately see that the equation can be solved for $c$ if and only if $r\in (-2,2)$, i.e.
$$
c: (-2,2)\to \R;\ r\mapsto \frac{\sin^{-1}(r/2)}{2r}
$$
and is therefore smooth and even. This finishes the construction of the smooth map $F:D_2\cp^n\to\cp^n\times\cp^n$. But is it a bijection? We can explicitly give an inverse. A point $(a,b)\in \cp^n\times\cp^n\setminus\bar\Delta$ defines a unique shortest geodesic $\gamma$ connecting $a=\gamma(0)$ to $b=\gamma(1)$. Now we set $x=\gamma(1/2)$ and $v=-j(a-b)$ in total we obtain
$$
F^{-1}(a,b)=(\gamma(1/2), -j(a-b)).
$$
Thus $F$ is a smooth bijection intertwining the moment maps. The construction of $F^{-1}$ in the 1-dimensional case is drawn in figure \ref{f2}.\\
 \begin{figure}
	\centering
 \includegraphics[width=0.4\textwidth]{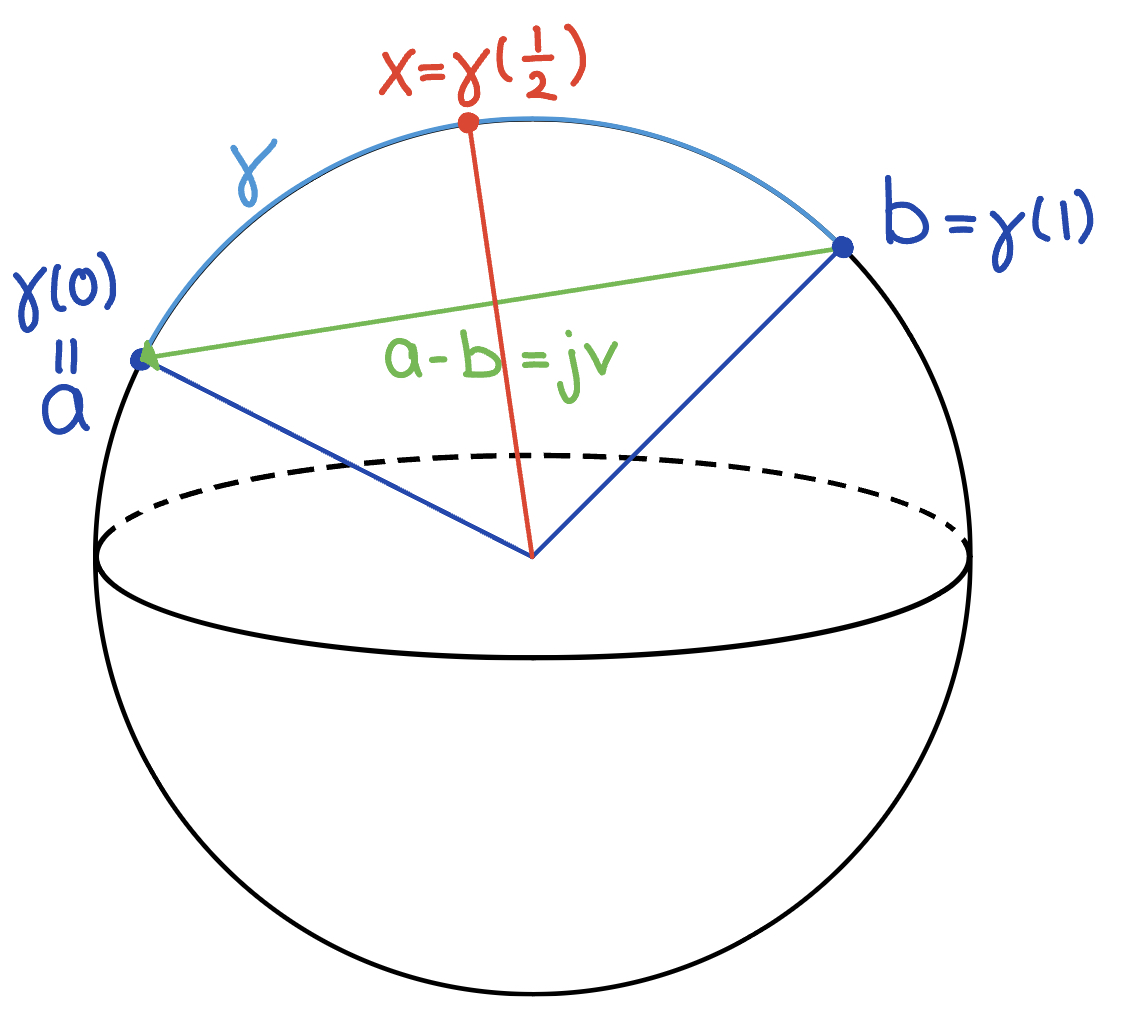}
	\caption{\textit{For any two points $a, b\in S^2$ (dark blue) that are not antipodal there is a unique shortest geodesic $\gamma$ (light blue) such that $\gamma(0)=a$ and $\gamma(1)=b$. We set $x:= \gamma(1/2)$ (red) and $jv=a-b$ (green). As $x$ is orthogonal to $a-b$ and thus orthogonal to $v=-j(jv)$ we find that indeed $v\in T_xS^2$, thus $F^{-1}$ is well defined. In higher dimensions the argument reduces to the 1-dimensional case as the symmetry group acts transitively on the totally geodesic copies of $\cp^1\cong S^2$.}}
    \label{f2}
\end{figure}
\ \\
Finally observe that the codimension of the orbits of $SU(n+1)$ is one. Thus any complement of $\mathcal{D}$ is 1-dimensional and therefore isotropic for both symplectic forms $\dd\lambda$ and $F^*(\omega_{FS}\ominus\omega_{FS})$. It follows from Lemma \ref{lem1} that $F$ is indeed a symplectomorphism. 
\end{proof}

\begin{Remark}
    For $(x,v)\in T\cp^n$ denote $(a,b):=F(x,v)\in\cp^n\times\cp^n$. Observe now that
    $$ F(H(x,v))=l\vert v\vert=l\vert jv\vert=l\vert \mu_1(x,v)\vert=l\vert \mu_2(F(x,v))=l\vert a-b\vert.$$
    We see that $H\circ F$ extends to a smooth map on $\cp^n\times\cp^n\setminus\Delta$ that is critical along $\bar\Delta$. We may therefore indeed interpret the map $F$ as a realization of the Lerman cut with respect to $H$ at the value $2l$.
\end{Remark}
\subsection{Real projective space}
Using the description of $\cp^n$ as adjoint orbit $O_Z\subset \mathfrak{su}(n+1)$ explained before, we can identify $\rp^n$ as an $\mathrm{SO}(n+1)\subset \mathrm{SU}(n+1)$ sub orbit 
$$
\R\mathrm P^n\cong \left \lbrace \mathrm{Ad}_g(Z)\vert \ \ g\in\mathrm{SO}(n+1)\right \rbrace \subset O_Z\cong \cp^n.
$$
It is fixed under the involution
$$
I:O_Z\to O_Z;\ p\mapsto p^T,
$$
as can be seen by direct computation
$$
(gZg^{-1})^T= g Z^T g^T = g Z g^{-1}\ \ \forall g\in\mathrm{SO}(n+1).
$$
The involution $I$ is antiholomorphic, because for all $v\in\mathfrak{p}$ we find
$$
\dd I_Z( j_Z(v))=[Z,v]^T=(Zv-vZ)^T=-(Zv^T-v^TZ)=-[Z,v^T]=-j_Z(\dd I_Z (v)).
$$
\begin{Lemma}
    The induced resp.\ subgroup action of $\mathrm{SO}(n+1)\subset \mathrm{SU}(n+1)$ on $(T\rp^n,\dd\lambda)$ resp.\ $(\cp^n, \omega_{FS})$ is Hamiltonian. The moment maps are given by 
    $$
    \mu_1:T\rp^n\to\mathfrak{so}(n+1);\ (x,v)\mapsto [x,v]
    $$
    resp.\
    $$
    \mu_2: \cp^n \to \mathfrak{so}(n+1);\ a\mapsto \mathrm{Re}(a).
    $$
\end{Lemma}
\begin{proof}
    The map $\mu_1$ is the restriction of the moment map for $(T\cp^n,\dd\lambda)$.
    We therefore only need to check that the map actually takes values in $\mathfrak{so}(n+1)$. From the description as suborbit it is only clear that $[x,v]\in\mathfrak{su}(n+1)$. The action of $\mathrm{SO(n+1)}$ on $T\R \mathrm{P}^n$ is transitive on the sphere subbundle, thus it is fine to check this at some point
    $$
    x=\frac{1}{n+1}\begin{pmatrix}
-i_{n-1} & \quad & \quad\\
\quad & -i & 0\\
\quad & 0 & ni
\end{pmatrix}\in \R\mathrm{P}^n,\ \ v=\frac{r}{2}\begin{pmatrix}
0_{n-1} & \quad & \quad\\
\quad & 0 & -i\\
\quad & -i & 0
\end{pmatrix}\in T_x\R \mathrm{P}^n\cong [x,\mathfrak{so}(n+1)].
    $$
Observe that $x=Z$ and $r=\vert v\vert$.
Indeed for this choice we see
$$
\overline{[x,v]}=[\bar x,\bar v]=[-x,-v]=[x,v]
$$
and therefore $[x,v]\in\mathfrak{so}(n+1)$. \\
\ \\
Next we look at the map $\mu_2$. Equivariance is clear as conjugation with a real matrix commutes with taking the real part. The rest of the proof is the following straight forward computation. Any tangent vector $v\in T_p\cp^n$ is of the form $v=\xi^\#_p$ for some $\xi\in\mathfrak{su}(n+1)$, thus for every $a\in \mathfrak{so}(n+1)$ we find
   \begin{align*}
       0&=\diff (\mathrm{Ad}_{e^{t\xi}}(p),\mathrm{Ad}_{e^{t\xi}} (a))\\
       &= \diff ( \mathrm{Ad}_{e^{t\xi}}(p), a)+(p,[\xi, a])\\
       &=\diff ( \mathrm{Re}(\mathrm{Ad}_{e^{t\xi}} (p)), a)+(p, [\xi, a])\\
       &=\dd (\mathrm{Re},a)(\xi^{\#})+\omega_{FS}(\xi^\#, a^\#)\\
       &=\dd (\mathrm{Re},a)(v)+\omega_{FS}(v, a^\#).
   \end{align*}
For the third equality we used that real and imaginary part are orthogonal with respect to the Killing form, i.e. for two real matrices $a,b$ such that $a, ib \in \mathfrak{su}(n+1)$ we know that 
$$
(a,ib)=-2\mathrm{tr} (a\cdot ib)=-2i\mathrm{tr} (a\cdot b)\in\R\ \ \Rightarrow\ \ (a, ib)=0.
$$
For the fourth equality we used that
$$
\omega_{FS}(\xi^\#,a^\#)=g(j\xi^\#,a^\#)=([p,\xi],a)=(p, [\xi,a]).
$$
\end{proof}
\noindent
We can now prove the following theorem analogously to Theorem \ref{thm1}.
\begin{Theorem}\label{thm2}
    There is an  $SO(n+1)$-equivariant symplectomorphism
    $$
    F: (D_{1}\R\mathrm{P}^n,\dd\lambda)\to (\cp^n\setminus Q^{n-1},\omega_{FS}),
    $$
    where $Q^{n-1}\subset\cp^n$ denotes the quadric 
    $$
    Q^{n-1}:=\lbrace [z_0:\ldots:z_n]\ \vert \ z_0^2+\ldots+z_n^2=0\rbrace\subset\cp^n. 
    $$
\end{Theorem}
\begin{proof}
    We make the following ansatz for the symplectomorphism realizing Theorem \ref{thm2},
$$
F: D_{1}\R\mathrm{P}^n\to\cp^n;\ \ (x,v)\mapsto \exp_x(-f(\vert v\vert)j_xv) = e^{f(\vert v\vert )v}xe^{-f(\vert v\vert)v},
$$
for some smooth, even function $f:(-1,1)\to \R$ depending on $r:=\vert v\vert$. This can be visualized nicely in dimension $n=1$ as shown in Figure \ref{fig4}.
\begin{figure}
	\centering
 \includegraphics[width=0.5\textwidth]{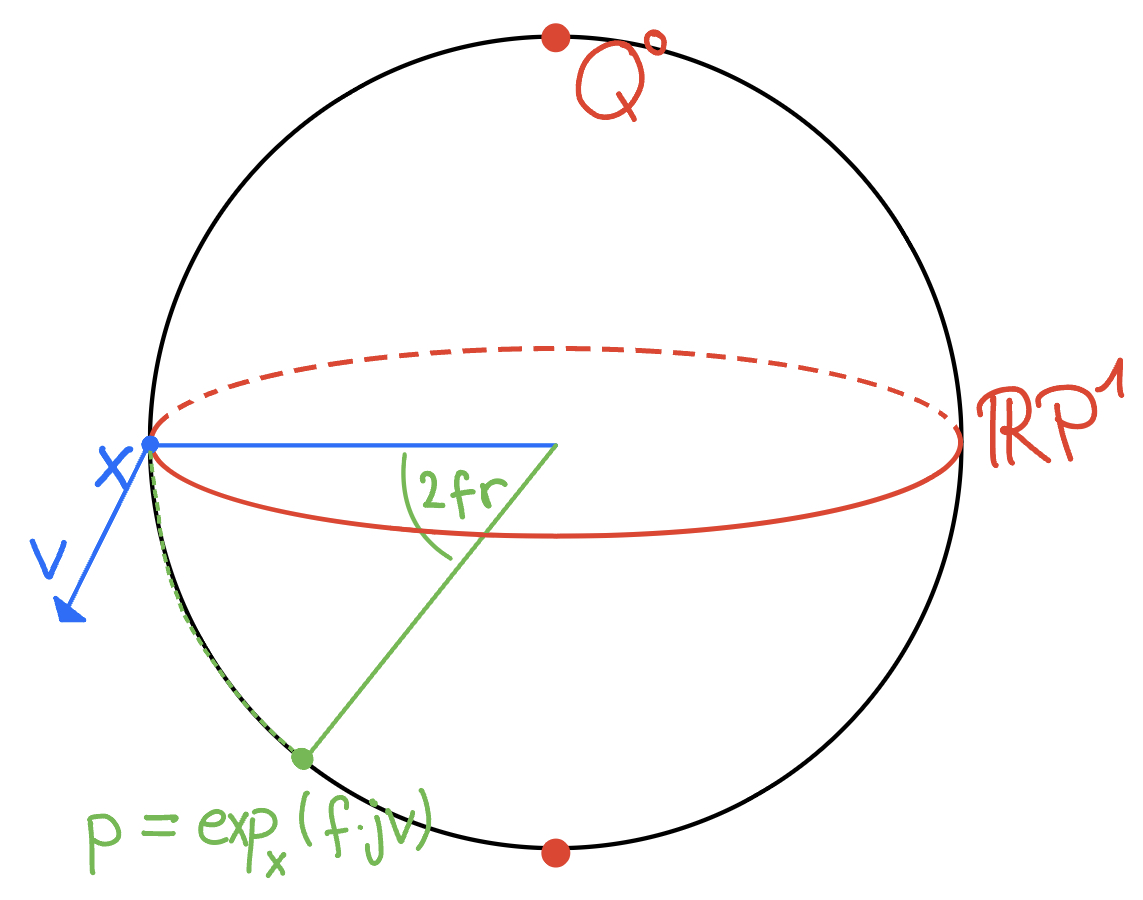}
	\caption{\textit{The figure shows a sketch of the map $F:D_{1}\rp^1\to\cp^1$. In dimension $n=1$ the quadric is given by the two antipodal points furthest away from $\rp^1$. This can for example be seen in homogenous coordinates, as $[1:\pm i]\in \cp^1$ are the only points solving the equation $z_0^2+z_1^2=0$. Now take a point $(x,v)\in T\rp^1$ (in blue) and follow the geodesic through $x$ in direction $jv$ until $p$ satisfies $jv=\mathrm{Re}(p)$ (in green). The angle between $p$ and $x$ determines $f$. In particular the image of $F$ coincides with $\cp^1\setminus Q^0$.}}
    \label{fig4}
\end{figure}
For the above choice of $(x,v)$ we find
$$
e^{f(r)v}=\begin{pmatrix}
1_{n-1} & \quad & \quad\\
\quad & \cos (fr) & i \sin (fr)\\
\quad & i\sin (fr) & \cos (fr)
\end{pmatrix}.
$$
It follows that
$$
e^{f(r)v} x e^{-f(r)v} =\frac{1}{n+1}\begin{pmatrix}
-i_{n-1} & \quad & \quad\\
\quad & i(n \sin(fr)^2-\cos (fr)^2) & -\frac{1}{2}(n+1) \sin (2fr)\\
\quad & \frac{1}{2}(n+1) \sin (2fr) & i(n \cos(fr)^2-\sin(fr)^2)
\end{pmatrix}
$$
and thus
$$
\mu_2(F(x,v))=\frac{1}{2}\begin{pmatrix}
0_{n-1} & \quad & \quad\\
\quad & 0 & - \sin (2fr)\\
\quad & \sin (2fr) & 0
\end{pmatrix}.
$$
On the other hand 
$$
\mu_1(x,v)=[x,v]=\frac{r}{2}\begin{pmatrix}
0_{n-1} & \quad & \quad\\
\quad & 0 & -1\\
\quad & 1 & 0
\end{pmatrix}.
$$
Comparing the matrix entries we find that $f$ must satisfy
\begin{equation}\label{ff}
\sin(2fr)=r,
\end{equation}
which is invertible for $r\in (-1,1)$ so that
$$
f(r)=\frac{\sin^{-1}(r)}{2r}
$$
is a smooth and even function.
\noindent
Thus it follows that $F$ is a smooth map intertwining the moment maps. But is it a bijection? We can explicitly give an inverse. A point $p\in \cp^n\setminus Q^{n-1}$ and its conjugate $I(p)$ are joined by a unique shortest geodesic $\gamma: [0, 1]\to \cp^n$. This can be seen as follows. If $p\notin \rp^n$,  $p$ and $I(p)$ lie in a unique copy of $\cp^1$. Now we look at Figure \ref{fig5} to see that there is a unique shortest geodesic joining $p$ and $I(p)$ if and only if $p\notin Q^{n-1}$. Thus
$$
F^{-1}(p)=(\gamma(1/2),-j\mathrm{Re}(p))
$$
defines the inverse.\\
\begin{figure}
	\centering
 \includegraphics[width=0.6\textwidth]{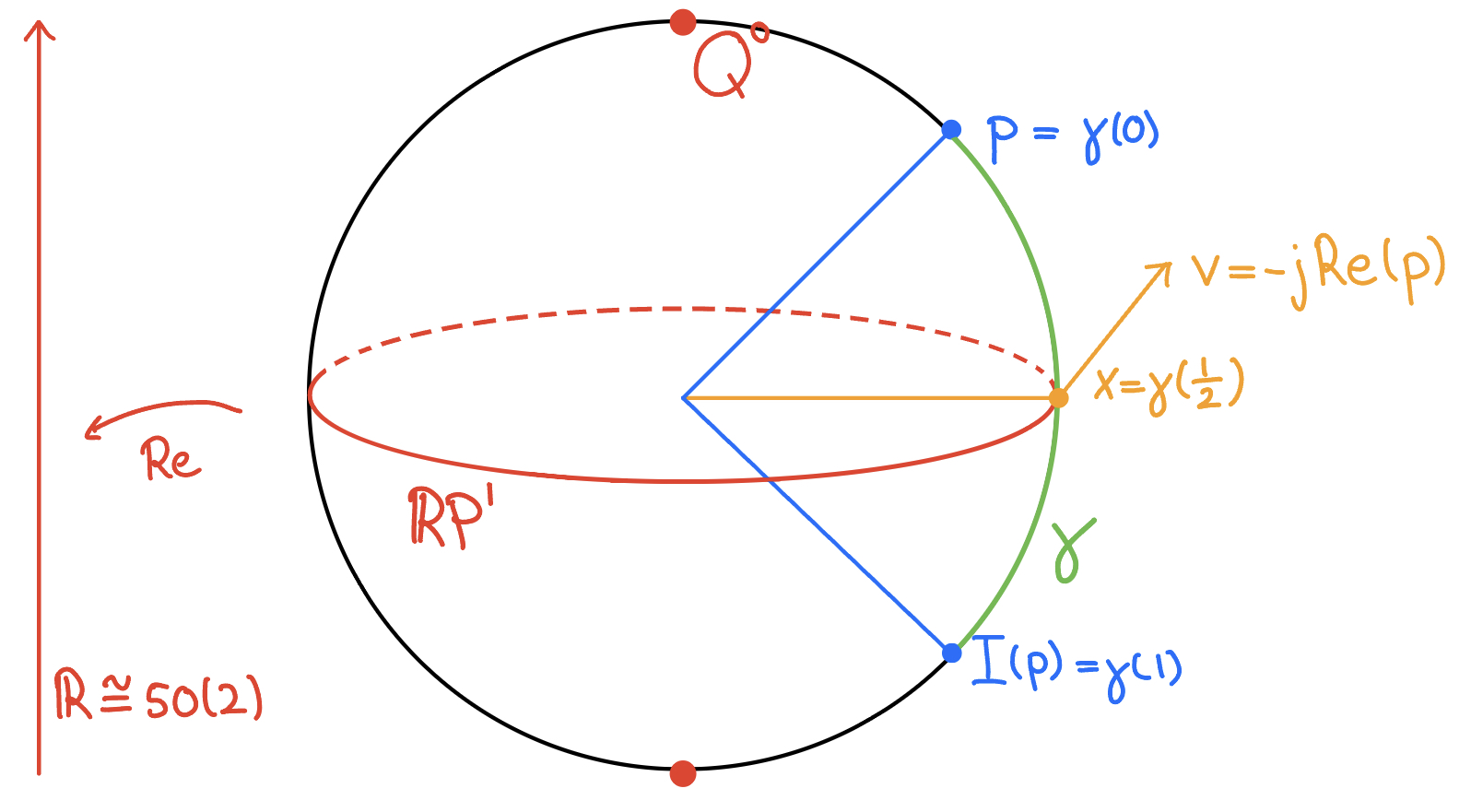}
	\caption{\textit{For any point $p\in \cp^1\setminus Q^0$ there is a unique shortest geodesic $\gamma$ (green) joining $p$ and its conjugate $I(p)$ (blue). Observe that taking the real part corresponds to projection onto the vertical axis (red). Now we set $x:=\gamma(1/2)$ and $v=-j\mathrm{Re}(p)$. The higher dimensional case reduces to this case as $\mathrm{SO}(n+1)$ acts transitively on the unit sphere bundle of $\rp^n$.}}
    \label{fig5}
\end{figure}
\ \\
\noindent
We see that $F$ is an equivariant smooth bijection that intertwines the moment maps. Again any complement of $\mathcal{D}$ is 1-dimensional and thus isotropic for any symplectic form. Thus $F$ fulfills the prerequisites of Lemma \ref{lem1} and therefore is the symplectomorphism realizing the symplectic identification in Theorem \ref{thm2}.
\end{proof}
\begin{Remark}
    For $(x,v)\in T\rp^n$ denote $p:=F(x,v)\in\cp^n$. Observe now that
    $$ H(F^{-1}(p))=l\vert v\vert=l\vert \mu_1(x,v)\vert=l\vert \mu_2(F(x,v))=l\vert \mathrm{Re}(p)\vert.$$
    We see that $H\circ F^{-1}$ extends to a smooth map on $\cp^n\setminus\rp^n$ that is critical along $Q^{n-1}$. We may thereforeindeed interpret the map $F$ as a realization of the Lerman cut with respect to $H$ at the value $l$.
\end{Remark}
\begin{Remark}
A closed almost complex manifold $(W,J)$ is called (compact) complexification of a real manifold $M$ if there exists an antiholomorphic involution
$$
I:W\to W;\ \ \ \dd I\circ J=-J\circ \dd I
$$
such that $M$ is isomorphic to the fixed point set of $I$. The symplectomorphisms constructed above realizes these complexifications of $\cp^n$ resp.\ $\rp^n$. The antiholomorphic involution on $(\cp^n\times\cp^n, j\ominus j)$ is given by
$$I: \cp^n\times\cp^n\to\cp^n\times\cp^n;\ \ (a,b)\mapsto (b,a).$$
The antiholomorphic involution on $(\cp^n, \omega_{FS})$ respresented as adjoint orbit is given by 
$$
I: \cp^n\subset\mathfrak{su}(n+1)\to \cp^n\subset\mathfrak{su}(n+1);\ \ p\mapsto p^T.
$$
Observe that in both cases the symplectomorphism is of the form
$$
DM\to W, (x,v)\mapsto \exp_x^W(-Jc(v)v),
$$
where in our cases $c$ is a scalar depending only on the norm of $v$. It seems likely that for general Kähler manifolds it is be possible to construct a similar map, maybe promoting $c$ to be operator valued and depending also on the direction of $v$ and possibly replacing the disc subbundle by a different neighborhood of the zero section. 
\end{Remark}
\section{Upper bound}\label{sec2}
To find an upper bound of the Hofer--Zehnder capacity, we will use a theorem by H. Hofer and C. Viterbo \cite[Thm.\ 1.16]{HV92} that shows existence of periodic orbits in the presence of holomorphic spheres. Let $(M,\omega)$ be a closed symplectic manifold and $J:TM\to TM$ a $\omega$-compatible almost complex structure. Consider the moduli space $\MM$ of $J$-holomorphic spheres $u: S^2\to M$ satisfying:
 \begin{align}\label{eq5}
 \begin{split}
     u(0)&\in\Sigma_0,\\
     u(\infty)&\in\Sigma_\infty,\\
     [u]&=A\\
     \frac{1}{2}\omega(A)&=\int_{\vert z\vert\leq 1}u^*\omega,
\end{split}
 \end{align}
for two closed disjoint submanifolds $\Sigma_0$ and $\Sigma_\infty$ and some homotopy class $A\in \pi_2(M)$. 
Observe that these conditions fix the reparametrization symmetry up to an $S^1$ (rotation of the complex plane). Further Hofer--Viterbo assume $A$ to be \emph{minimal}, i.e.\
$$
\omega(A)=\inf\lbrace \omega(B)\ \vert B\in [S^2, M],\ \omega(B)>0\rbrace,
$$
which ensures that bubbling cannot occur. We replace this condition by requiring $A$ to be $\omega$-\textit{indecomposable}, i.e.\ $A$ can not be written as sum $\sum_i B_i$ for free homotopy classes $B_i\in \pi_2(M)$ with $\omega(B_i)>0$. The reader my check that the proof of Theorem 1.16 in \cite{HV92} goes through without any changes also for $\omega$-indecomposable classes.\\
\ \\
If $J$ is regular the moduli space $\MM$ is a compact oriented manifold without boundary and thus defines an oriented cobordism class. Since the $S^1$ reparametrization symmetry of pseudoholomorphic curves is not fixed, the moduli space inherits a free $S^1$-action and one considers cobordisms of compact, oriented, free $S^1$-manifolds. 
\begin{Theorem}[Hofer--Viterbo, Thm.\ 1.16 \cite{HV92}]\label{thm9}
Let $(M,\omega)$ be a closed symplectic manifold, $A\in \pi_2(M)$ $\omega$-indecomposable, $J$ a smooth, regular, $\omega$-compatible almost complex structure and $\Sigma_0, \Sigma_\infty$ two disjoint nonempty closed submanifolds of $M$. Suppose $H : M \to \R$ is a Hamiltonian such that 
$$
H\vert_{\mathcal{U}(\Sigma_0)}=\min(H)\ \ \text{and}\ \ H\vert_{\mathcal{U}(\Sigma_\infty)}=\max(H),
$$
then, if the free $S^1$-cobordism class $[\mathcal{M}]\neq 0$, the Hamiltonian system $\dot x = X_H(x)$ possesses a non constant contractible $T$-periodic solution with
$$
0<T(\max (H)-\min (H))<\omega(A).
$$
In particular 
$$
c_{HZ}(M\setminus\Sigma_\infty,\Sigma_0,\omega)\leq\omega(A)
$$
and if $\Sigma_0$ is a one point set we can conclude
$$
c_{HZ}(M\setminus\Sigma_\infty,\omega)\leq\omega(A).
$$
\end{Theorem}
\noindent
We can now use this theorem and the compactifications constructed in the previous section to determine an upper bound of the Hofer--Zehnder capacity of $(D_1\cp^n,\dd\lambda)$ resp. $(D_1\rp^n,\dd\lambda)$.\\
\ \\
For the case $(M,\omega)=\overline{(D_2\cp^n,\dd\lambda)}\cong(\cp^n\times\cp^n,\omega\ominus\omega)$ we pick $\Sigma_0=\lbrace (p,q)\rbrace $ a one point set and $\Sigma_\infty$ the anti-diagonal divisor $\bar{\Delta}$. We fix $A\in [S^2,\cp^n\times \cp^n]$ to be the class of a degree one curve in the first factor. Further by \cite[Prop.\ 7.4.3]{DS17} the split complex structure $j\ominus j$ is regular. Now every holomorphic sphere through $(p,q)$ in class $A$ is of degree 1 in the first and constant in the second factor. Further $\bar\Delta$ restricted to the fiber over $q$ is a copy of $\cp^{n-1}$, in particular any $j\ominus j$-holomorphic sphere through $(p,q)$ in class $A$ intersects $\bar\Delta$. This means the moduli space of holomorphic curves $\mathcal{M}$ through $(p,q)$ in class $A$ can be identified with the space of complex lines through $p$. Recall that we left the $S^1$-reparametrization symmetry, therefore $\mathcal{M}\cong S^{2n-1}$ and not $\mathcal{M}\cong \cp^{n-1}$ and the free circle action given by rotating the holomorphic spheres parameterizes the Hopf fibration. 
\begin{Lemma}
    The oriented $S^1$-cobordism class of $S^{2n-1}$ does not vanish.
\end{Lemma}
\begin{proof}
    We need to show that there is no oriented free $S^1$-manifold with boundary $S^{2n-1}$. Assume there was, we could then glue in the standard ball $B\subset \C^n$ of radius $1$ to obtain a closed orientable manifold with a circle action with exactly one fixed point. This is because multiplication with $e^{it}$ yields such an action on $B\subset \C^n$. Now such a manifold does not exist as shown in \cite[Lem.\ 2.1]{20}.
\end{proof}
\noindent
We can use Theorem \ref{thm9} to find
$$
c_{HZ}(D_2\cp^n,\dd\lambda)=c_{HZ}(\cp^n\times\cp^n\setminus\bar\Delta,\omega\ominus\omega)\leq \omega(A)=4\pi=2l.
$$
It follows that 
\begin{equation}\label{e1}
c_{HZ}(D_1\cp^n,\dd\lambda)\leq l.
\end{equation}
\ \\
\noindent
For the case $(D_1\rp^n,\dd\lambda)$ we need to work less as Hofer--Viterbo already computed 
\begin{equation}\label{e2}
c_{HZ}(D_{1}\rp^n,\dd\lambda)\leq c_{HZ}(\cp^n,\omega_{FS})=4\pi=2l
\end{equation}
using Theorem \ref{thm9}. In the next section we will see that the bounds \eqref{e1} and \eqref{e2} are actually good enough. 
\section{Lower bounds}\label{sec3}
Finding a lower bound is in some sense more elementary and much more hands on. Indeed we will simply construct explicit admissible Hamiltonians. We observed already that the Hamiltonian $H(x,v)=l\vert v\vert$ generates a Hamiltonian circle action away from the zero-section. Composing $H$ with a smooth even function $f:[-l,l]\to \R$ that satisfies $0\leq f'< 1$, $f\equiv 0$ near $0$,  $f\equiv l-\varepsilon$ near $\pm l$ for an arbitrary small $\varepsilon>0$, yields an admissible function $\Tilde{H}:=f\circ H$ that extends smoothly to the zero-section. Further the oscillation of $\Tilde{H}$ is $l-\varepsilon$ so that in the limit $\varepsilon\to 0$ we obtain $l$ as lower bound. This finishes the computation of $$c_{HZ}(D_1\cp^n,\dd\lambda)=l$$ as the lower bound agrees with the upper bound derived in \eqref{e1}.\\
\ \\
\noindent
However in the case of $\rp^n$ this lower bound is not good enough. To improve it we look at the dynamics for a family of Hamiltonians 
    $$
    H_\varepsilon(x,v)=\sqrt{\Vert v\Vert_x^2+V_\varepsilon(x)}
    $$
    with potentials $V_\varepsilon: \rp^n\to\R$ that approximate billiard dynamics on a hemisphere as $\varepsilon\to 0$. We may lift the potential to $S^n$ and assume that $V_\varepsilon(-x)=V_\varepsilon(x)$ so that the potential descends to a function $V_\varepsilon$ on $\rp^n$. 
    We fix a point $N\in S^n$. The isotropy subgroup of $SO(n+1)$ that fixes $N$ is isomorphic to $SO(n)$ and leaves the equatorial copy of $S^{n-1}$ invariant. 
    We want the family $V_\varepsilon$ to be rotation invariant i.e.\ $V_\varepsilon(R(x))=V_\varepsilon(x)$ for all $R\in \mathrm{SO(n)}$. This implies that $V_\varepsilon(x)$ only depends on the distance of $x$ from $N$. In particular the gradient $\nabla V_\varepsilon (x)$ points along the geodesic connecting $x$ and $N$. 
    \begin{Lemma}\label{lem2}
           The Hamiltonian vector field for $H_\varepsilon$ with respect to $\dd\lambda$ is given by
            $$
            (X_{H_\varepsilon})_{(x,v)}=\frac{1}{H_\varepsilon} \left( X_{(x,v)}-\frac{1}{2}(\nabla V_\varepsilon(x))^\VV\right),
            $$
             where $X=X_E$ is the vector field generating geodesic flow and $(\cdot)^\VV:TS^n\to TTS^n$ denotes the vertical lift. 
    \end{Lemma}
    \begin{proof}
        The kernel of the differential of the footpoint projection $\pi:TS^n\to S^n$ defines the so called vertical distribution $\VV\subset TTS^n$. Now the Levi-Civita connection determines a complement $\HH\subset TTS^n$ called the horizontal distribution. For any $(x,v)\in TS^n$ we can identify $T_xS^n$ with $\HH_{(x,v)}$ and $\VV_{(x,v)}$ using the so called horizontal and vertical lift respectively. The horizontal lift is defined by
        $$
           T_xS^n\to\HH_{(x,v)};\ \ w\mapsto w^\mathcal{H}:=\diff (\gamma(t),P_\gamma v)
        $$
        where $\gamma:(-\varepsilon,\varepsilon)\subset\R\to M$ is a smooth curve satisfying $\gamma(0)=x$ and $\dot\gamma(0)=w$ and $P_\gamma$ denotes parallel transport along $\gamma$. The vertical lift is defined by
        $$
        T_x S^n\to\VV_{(x,v)}; \ \ w\mapsto w^\mathcal{V}:=\diff (x, v+tw).
        $$
        These lifts can be used to lift any vector field on $S^n$ to a vector field on $TS^n$. Observe that per definition $\lambda$ vanishes on $\VV$, further as proven in \cite[Lemma 2]{D62}, $[A^\VV,B^\HH]\subset \VV$ and $[A^\VV,B^\VV]=0$ for all vector fields $A,B$ on $S^n$. We conclude that 
        \begin{align*}
            \dd\lambda(A^\VV,B^\HH)&=A^\VV(\lambda(B^\HH))-B^\HH(\lambda(A^\VV))-\lambda([A^\VV, B^\HH])\\
            &=A^\VV(\lambda(B^\HH))=A^\VV(g(v,B))=g(A,B)
        \end{align*}
        and 
        $$
        \dd\lambda(A^\VV,B^\VV)=A^\VV(\lambda(B^\VV))-B^\VV(\lambda(A^\VV))-\lambda([A^\VV, B^\VV])=0.
        $$
        Thus $\dd\lambda(A^\VV, \cdot )=g(A,\dd\pi\cdot)$. We use this to verify the formula for $X_{H_\varepsilon}$,
    \begin{align*}
        \dd\lambda(X_{H_\varepsilon},\cdot)&=\frac{1}{H_{\varepsilon}}\left(\dd\lambda(X,\cdot)-\frac{1}{2}\dd\lambda((\nabla V_\varepsilon)^\VV,\cdot)\right)\\ 
        &=-\frac{1}{H_{\varepsilon}}\left(\dd E+\frac{1}{2}g(\nabla V_\varepsilon,\dd\pi \cdot)\right)
        =-\frac{1}{2H_{\varepsilon}}\left(2\dd E+\dd V_\varepsilon\right)=-\dd H_\varepsilon.
    \end{align*}       
    \end{proof}
    \noindent
    The Lemma shows in particular that $(X_{H_\varepsilon})_{(x,v)}\in TTS^2$ for the totally geodesically embedded copy of $S^2$ through $x$ and $N$ tangent to $v$ and therefore its flow also preserves this copy of $TS^2$. We reduced the problem to the situation in dimension $n=2$. Now we introduce coordinates on this copy of $S^2$ to describe the dynamics explicitly.
    We use $\theta\in [0,\pi]$ and $\varphi\in \R/ 2\pi \Z$. Where $N$ corresponds to $\theta=0$. Set
    $$
    V_\varepsilon(x)=\left\{%
\begin{array}{ll}
    1, & \hbox{for}\ \theta(x)\in [\pi/2-\varepsilon/2,\pi/2] \\
    0, & \hbox{for}\ \theta(x)\in [0,\pi/2-\varepsilon]\\
\end{array}%
\right.
    $$
and interpolate smoothly for $\theta(x)\in [\pi/2-\varepsilon,\pi/2-\varepsilon/2]$ so that $V_\varepsilon\geq 0$ and $V_\varepsilon'\geq 0$. Now extend $V_\epsilon$ according to the required symmetries, i.e. $V_\varepsilon(x)=V_\varepsilon(-x)$ and $V_\varepsilon(Rx)=V_\varepsilon(x)$ for all $R\in \mathrm{SO}(n)$. 
In this coordinates the Hamiltonian reads
$$
H_\varepsilon(x,v)=\sqrt{\dot\theta^2+\sin(\theta)^2\dot\varphi^2+V_\epsilon(\theta)},
$$
where by abuse of notation we write $\dot \varphi, \dot \theta$ instead of $v_\varphi, v_\theta$ and $V_\varepsilon(\theta)$ instead of $V_\epsilon(x)$ as the potential only depends on $\theta$.
Observe that $H_\varepsilon$ does not depend on $\varphi$, this implies immediately that $\sin(\theta)^2 \dot\varphi$ is preserved. We conclude that the maximal angular velocity $\dot\varphi_{\max}$ corresponds to the minimal radius $\theta_{\min}$. We want to find upper bounds for the periods of all orbits for the Hamiltonian $H_\varepsilon$ as $\varepsilon$ tends to zero. Therefore we take an arbitrary orbit $\gamma(t)=(\theta(t), \varphi(t),\dot\theta(t),\dot\varphi(t))$ with energy $H_\varepsilon(\gamma)< 1$. We split into two cases: $\theta_{\min}\geq \frac{\pi}{2}-\sqrt{\varepsilon}$ and $\theta_{\min}\leq \frac{\pi}{2}-\sqrt{\varepsilon}$. \\
\ \\
\noindent
\textbf{Case $\theta_{\min}\geq \frac{\pi}{2}-\sqrt{\varepsilon}$\ :}
At $\theta(t)=\theta_{\min}$ we have $\dot\theta=0$, thus 
$$
1> H_\varepsilon=\sqrt{\sin(\theta_{\min})^2\dot\varphi_{\max}^2+V_\epsilon(\theta_{\min})}\geq \sin(\theta_{\min})\dot\varphi_{\max}\Rightarrow \dot\varphi_{\max}\leq\frac{1}{\sin(\theta_{\min})}.
$$
Clearly the period $T$ must satisfy
$$
T> \frac{2\pi}{\dot\varphi_{\max}}\geq 2\pi\sin(\theta_{\min})\geq 2\pi \cos(\sqrt{\varepsilon})\xlongrightarrow{\varepsilon\longrightarrow 0} 2\pi.
$$
\noindent
\textbf{Case $\theta_{\min}\leq \frac{\pi}{2}-\sqrt{\varepsilon}$\ :}
For small $\varepsilon$ certainly $\varepsilon<\sqrt{\varepsilon}$, so $\theta_{\min}\leq \frac{\pi}{2}-\sqrt{\varepsilon}$ implies $V(\theta_{\min})=0$. In particular segments of the orbit are geodesics parametrized by unit speed. There are at least two of these segments as by symmetry of the potential the orbit must leave the region where $V_\varepsilon\neq 0$ with the same angle it enters the region. The  length of these segments is $\pi-2 l_\varepsilon$, where $l_\varepsilon$ is the length of the geodesic segment contained in the region $\pi/2-\varepsilon\leq\theta\leq \pi/2$ where the potential $V_\varepsilon$ does not vanish as depicted in Figure \ref{fig7}. Now we can bound the period from below as follows
$$
T> 2\pi-4l_\varepsilon\geq 2\pi(1-\sin (l_\varepsilon))=2\pi\left(1-\frac{\sin(\varepsilon)}{\sin(\pi/2-\theta_{\min})}\right)\geq 2\pi\left(1-\frac{\sin(\varepsilon)}{\sin(\sqrt{\varepsilon})}\right)\xlongrightarrow{\varepsilon\longrightarrow 0} 2\pi.
$$
The bound on $\sin(l_\varepsilon)$ is obtained using the spherical law of sines 
$$
\frac{\sin(l_\varepsilon)}{\sin(\pi/2)}=\frac{\sin(\varepsilon)}{\sin(\pi/2-\theta_{\min})}
$$
as explained in Figure \ref{fig7}.\\

\begin{figure}
	\centering
 \includegraphics[width=1\textwidth]{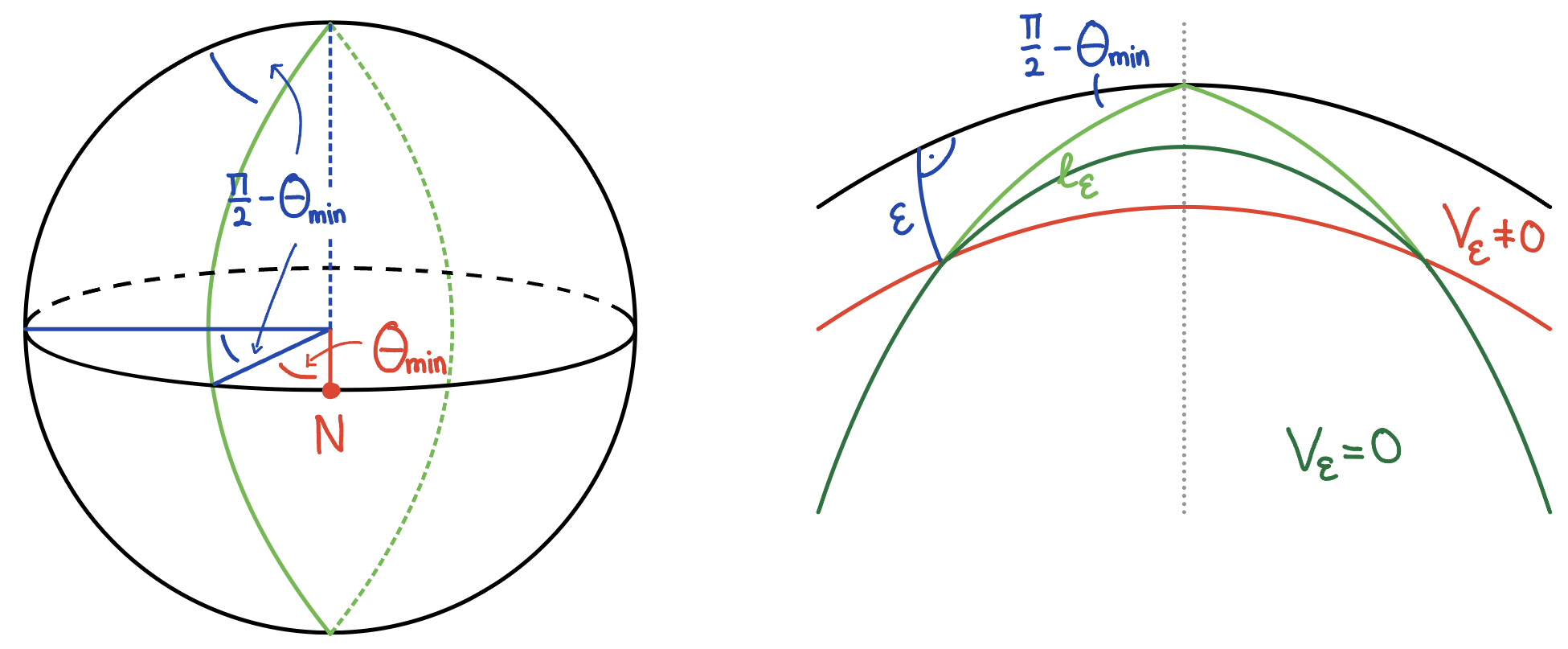}
	\caption{\textit{On the left we see (in green undashed) a geodesic segment on the hemisphere centered at $N$ (in red). Indeed it hits the equator with angle $\frac{\pi}{2}-\theta_{\min}$ (in blue). On the right we see (in dark green) a sketch of how the trajectory might continue when it enters the region where the potential does not vanish. By conservation of angular momentum the picture must be symmetric with respect to the grey dotted line. In particular the trajectory leaves the region with the same angle entering it. Further we see that the segment of the geodesic $l_\varepsilon$ is the hypotenuse of a right spherical triangle. The length of the side opposite to $\frac{\pi}{2}-\theta_{\min}$ is $\varepsilon$ and thus we can determine $l_\varepsilon$ using the spherical law of sines.}}
    \label{fig7}
\end{figure}

\ \\
\noindent
In total we conclude that for small enough $\varepsilon>0$ the period is bounded by
$$
T> 2\pi (1-\sqrt{\varepsilon})
$$
and therefore the orbits for the scaled Hamiltonian $(1-\sqrt{\varepsilon})2\pi H_\varepsilon$ have periods $T\geq 1$. Now we can, analogously to the modifications for $H(x,v)=l\vert v\vert$, use a function $f:[-1,1]\to [0,\infty)$ changing the oscillation arbitrary small so that $f\circ H_\varepsilon$ is admissible. In particular 
$$ c_{HZ}(D_1\rp^n,\dd\lambda)\geq (1-\sqrt{\varepsilon})2\pi\mathrm{osc}(H_\varepsilon)=(1-\sqrt{\varepsilon})2\pi\ \ \forall\varepsilon>0, $$
which finishes the proof of Theorem \ref{thma} as $2\pi=2l$ for the normalization\footnote{As $2\pi$ is the length of a prime geodesic on $S^n$, $\pi$ is the length of a prime geodesic on $\rp^n$.} we worked with and this lower bound agrees with the upper bound derived in \eqref{e2}.
\section{A magnetic version} \label{sec4}
In this section we consider the tangent bundle of $\cp^n$ with the twisted symplectic form $\dd\lambda-s\pi^*\omega_{FS}$, where $\omega_{FS}\in \Omega^2(\cp^n)$ denotes the Fubini-Study form while $s\in \R$ is a parameter representing the strength of the magnetic twist. We will compute the Hofer--Zehnder capacity for the unit disc subbundle using the exact same strategy as in the proof of Theorem \ref{thma}. Indeed the magnetic version is a generalization as the case $s=0$ recovers the untwisted case.
\subsection{Compactification}
Observe that $\omega_{FS}$ is invariant under the $\mathrm{SU}(n+1)$-action and therefore the induced action of $\mathrm{SU}(n+1)$ on $(T\cp^n,\dd\lambda-s\pi^*\omega_{FS})$ remains symplectic. Indeed it is not too hard to check that 
\begin{equation}
    \mu_1: T\cp^n\to\mathfrak{su}(n+1);\ (x,v)\mapsto [x,v]-sx
\end{equation}
is a moment map. The term $-sx$ compensates the twist as the inclusion of $\cp^n\hookrightarrow \mathfrak{su}(n+1)$ as adjoint orbit is a moment map for $\omega_{FS}$. We can use the same strategy as before to show that the symplectomorphism from Theorem \ref{thm1} can be adapted to the twisted situation.
\begin{Theorem}\label{thm3}
There is a symplectomorphism, which is equivariant with respect to the Hamiltonian $\mathrm{SU}(n+1)$ actions, 
$$
F: (D_1 \cp^n,\dd\lambda-s\pi^*\omega_{FS})\to (\cp^n\times \cp^n\setminus \bar\Delta, R_1\omega_{FS}\ominus R_2\omega_{FS}),
$$
where $\bar\Delta\subset \cp^n\times \cp^n$ denotes the anti-diagonal divisor
$$
\bar\Delta=\lbrace (p,q)\in\cp^n\times\cp^n\ \vert \ \mathrm{dist}(p,q) \ \text{maximal}\rbrace
$$
and 
$$
R_1=\frac{1}{2}\left (\sqrt{s^2+1}+s\right),\ \ R_2=\frac{1}{2}\left (\sqrt{s^2+1}-s\right).
$$
\end{Theorem}
\begin{proof}
Compared to the proof of Theorem \ref{thm1} we modify the Ansatz a little by considering
$$ F: D_1 \cp^n \to \cp^n\times \cp^n\setminus\bar\Delta;\ \ \ (x,v)\mapsto \left(\exp_x(-c_1(r)j_xv),\exp_x(c_2(r)j_xv)\right)
$$
for two different functions $c_1, c_2$ that will be determined imposing the relation 
$$
\mu_1(x,v)=\mu_2(F(x,v)).
$$
Per construction the moment maps and the map $F$ are equivariant. Hence it is again enough to check the moment map triangle at some point
$$
x=Z=\frac{1}{n+1}\begin{pmatrix}
-i_{n-1} & \quad & \quad\\
\quad & -i & 0\\
\quad & 0 & ni
\end{pmatrix},\ \ v=\frac{r}{2}\begin{pmatrix}
0_{n-1} & \quad & \quad\\
\quad & 0 & -i\\
\quad & -i & 0
\end{pmatrix} \in T_Z\cp^n,\ r=\vert v\vert>0,
$$
as the group action is transitive on the sphere subbundle of $T\cp^n$.
We compute 
$$
\mu_{D_\rho \cp^n}(x,v)=[x,v]-(R_1-R_2)x=\frac{r}{2}\begin{pmatrix}
0_n & \quad & \quad\\
\quad & 0 & 1\\
\quad & -1 & 0
\end{pmatrix}-\frac{R_1-R_2}{n+1}\begin{pmatrix}
-i_{n-1} & \quad & \quad\\
\quad & -i & 0\\
\quad & 0 & ni
\end{pmatrix}.
$$
Recall that the geodesic through $x$ in direction $j_x v$ is given by
$$
\gamma_{(x,j_xv)}(t)=\mathrm{Ad}_{\exp(-tv)}x=\frac{1}{n+1}\begin{pmatrix}
-i_{n-1} & \quad & \quad\\
\quad & -i(\cos(rt)^2-n\sin(rt)^2) & (n+1)\sin(rt)\cos(rt)\\
\quad & -(n+1)\sin(rt)\cos(rt) & i(n\cos(rt)^2-\sin(rt)^2)
\end{pmatrix}.
$$
It parametrizes an affine circle in $\mathfrak{su}(n+1)$ since
$$
V(t):=\frac{1}{2}\left(\gamma_{(x,j_xv)}(t)-\gamma_{(x,j_x v)}(t+\frac{\pi}{2r})\right)=\frac{1}{2}\begin{pmatrix}
0_{n-1} & \quad & \quad\\
\quad & -i\cos(2rt) & \sin(2rt)\\
\quad & -\sin(2rt) & i\cos(2rt)
\end{pmatrix}
$$
satisfies $\vert V\vert^2=-2\mathrm{tr}(V^2)=1$.
The circle is centered at
$$
y:=\frac{1}{2}\left(\gamma_{(x,j_x v)}(t)+\gamma_{(x,j_x v)}(t+\frac{\pi}{2r})\right)=\frac{1}{n+1}\begin{pmatrix}
-i_{n-1} & 0 & 0\\
0 & \frac{n-1}{2}i & 0\\
0 & 0 & \frac{n-1}{2}i
\end{pmatrix}.
$$
Observe that $\gamma_{(x,j_xv)}(t)=V(t)+y$.
\noindent
Now the second moment map is given by
$$
\mu_2(F(x,v))=R_1\gamma_{(x,j_xv)}(-c_1)-R_2\gamma_{(x,j_xv)}(c_2)=(R_1-R_2)y+R_1V(-c_1)-R_2V(c_2).
$$
Comparing the coefficients of the two matrices yields two equations for $c_1$ and $c_2$:
   \begin{equation}\label{eq:22}
   \begin{aligned}
    R_1\sin(2c_1r)+R_2\sin(2c_2r)&=r, \\
    R_1\cos(2c_1r)-R_2\cos(2c_2r)&=R_1-R_2.
   \end{aligned}
   \end{equation}
The moment map triangle commutes if and only if these two equations hold. There is an elementary geometric interpretation of the equations \eqref{eq:22} which comes from the fact that geodesics in $\cp^n\subset\mathfrak{su}(n+1)$ restrict to round circles in an affine copy of $\R^2\subset\mathfrak{su}(n+1)$ and is shown in Figure \ref{fig3} and Figure \ref{fig9}. From figure \ref{fig6} we deduce that a solution exists if and only if $r:=\vert v\vert\leq 1$.
\begin{figure}[h!]
	\centering
  \includegraphics[width=0.7\textwidth]{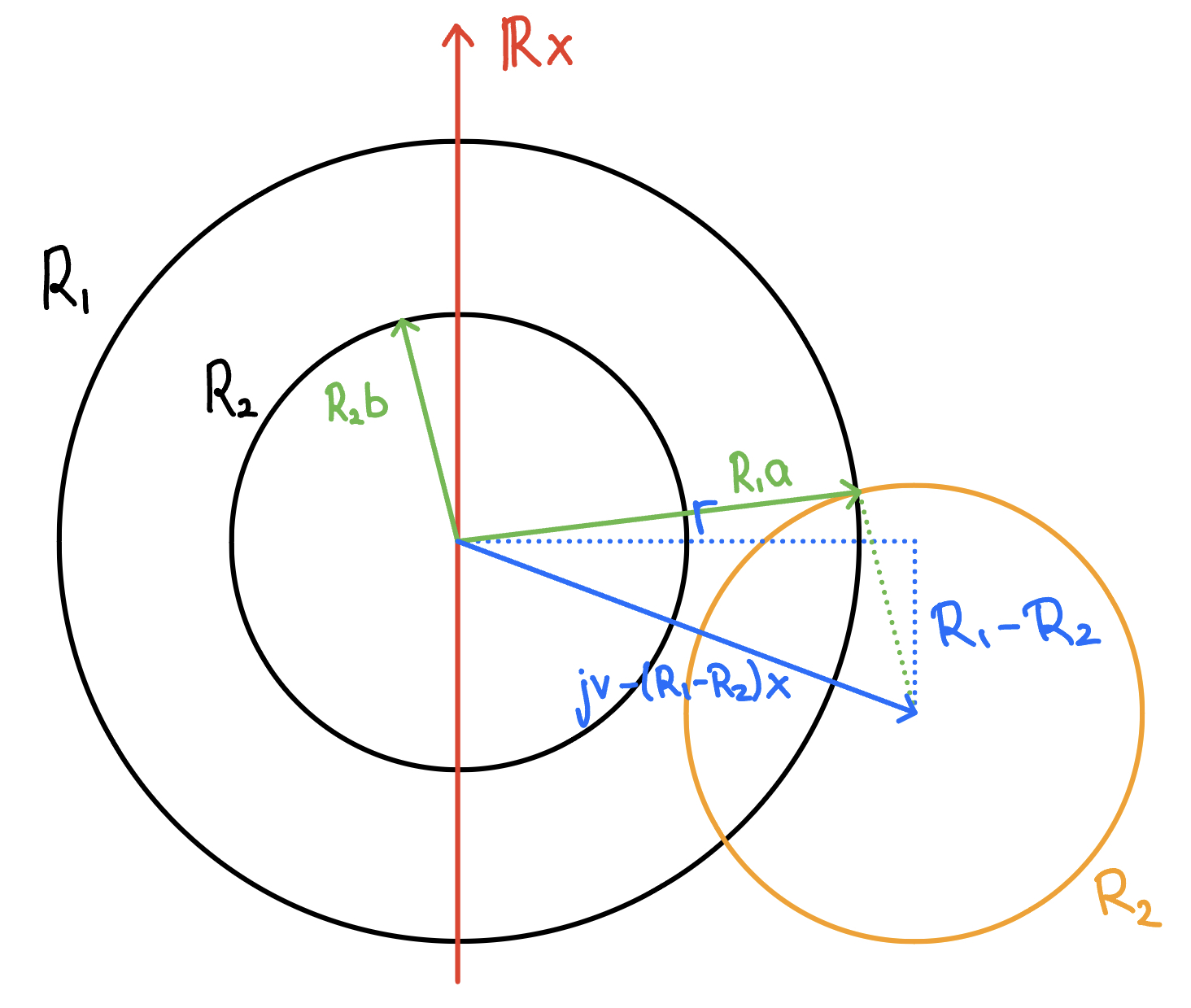}
	\caption{\textit{The geodesic $\gamma_{(x,jv)}$ parametrizes a round circle. We draw this circle twice (in black), one of radius $R_1$ and one of radius $R_2$. The moment map $jv-(R_1-R_2)x$ (in blue) must be equal to $R_1a-R_2b$. Thus we need to decompose $jv-(R_1-R_2)x$ into two vectors (in green), one of length $R_1$ the other of length $R_2$. This is done with an auxiliary circle (in orange) of radius $R_2$ centered at $jv-(R_1-R_2)x$. In principle there are two solutions (intersections of the orange and the outer black circle). Counting counterclockwise from the $x$-axis we pick the first. This is equivalent to demanding $2r(c_1+c_2)\leq \pi$.}}
    \label{fig3}
\end{figure}

\begin{figure}[h!]
	\centering
  \includegraphics[width=0.9\textwidth]{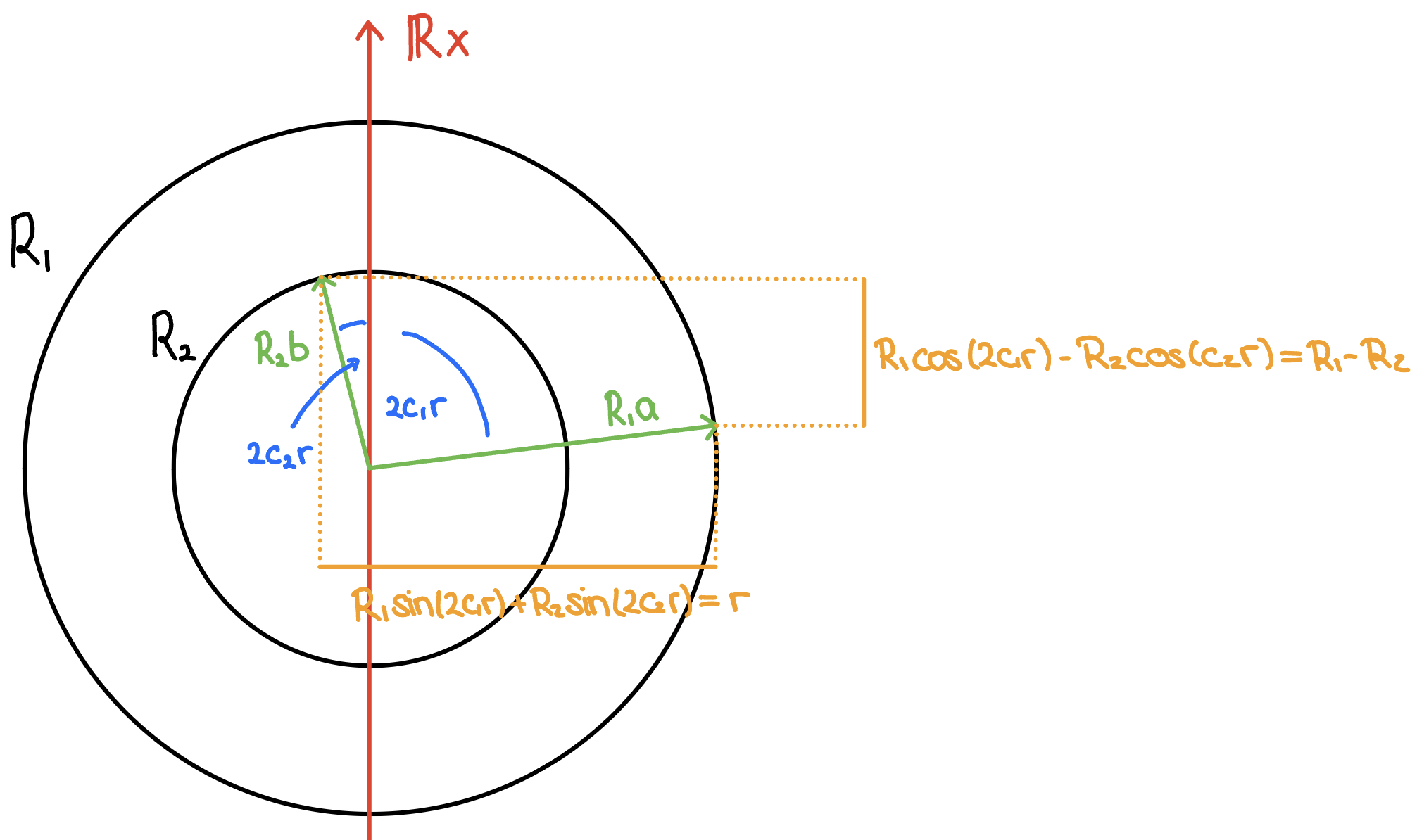}
	\caption{\textit{To see that the configuration drawn is a geometric interpretation of the equations \eqref{eq:22} denote the angle between the $x$-axis and the vectors by $2c_1r$ and $2c_2r$ (in blue). Now projecting the difference $R_1a-R_2b$ to the horizontal respectively the vertical axis yields the equations \eqref{eq:22} (in orange).}}
    \label{fig9}
\end{figure}

\begin{figure}[h!]
	\centering
  \includegraphics[width=0.7\textwidth]{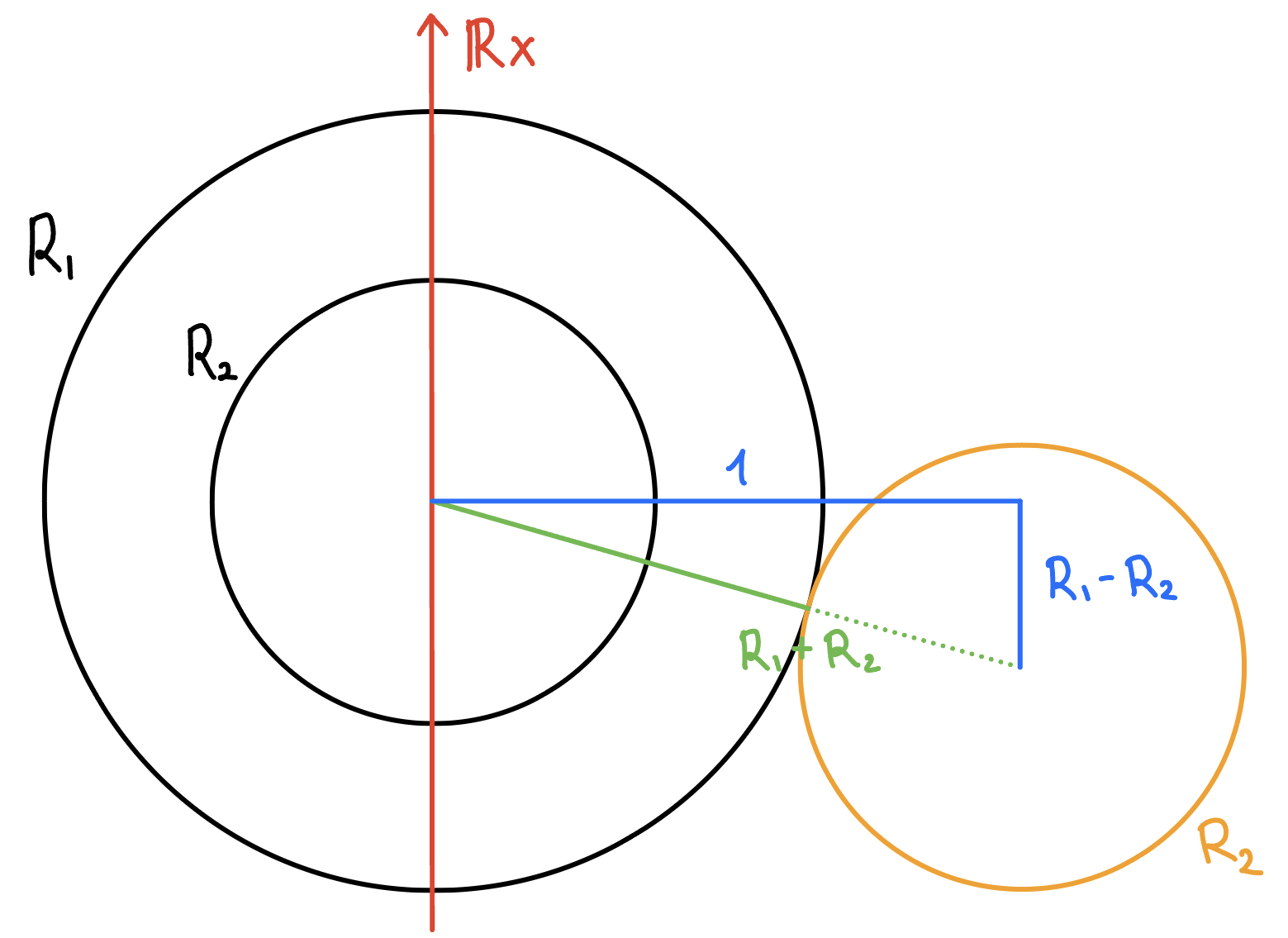}
	\caption{\textit{The figure shows the boundary case where only one intersection exists. Indeed it occurs when $r=1$ as $(R_1+R_2)^2=1+(R_1-R_2)^2$ using that by the definition of $R_1$ and $R_2$, the product $4R_1R_2=1$. It follows that there are two solutions to the equations \eqref{eq:22} if $0<r<1$. Demanding further that $2r(c_1+c_2)<\pi$ we fix one of the solutions.}}
    \label{fig6}
\end{figure}
\begin{Lemma}
The functions 
$$
c_1, c_2:\ (-1,1)\to (0,\infty)
$$
implicitly defined via the equation \eqref{eq:22} are smooth and even.
\end{Lemma}
\begin{proof}
First observe that the defining equations \eqref{eq:22} are invariant under $r\to -r$, thus $c_1$ and $c_2$ must be even. We will prove smoothness applying the implicit function formula to the map
$$
G: (-1,1)\times (0,\infty)^2\to \R^2;\qquad (r, c_1,c_2)\to \begin{pmatrix}
R_1\sin(2c_1r)+R_2\sin(2c_2r)-r, \\
R_1(\cos(2c_1r)-1)-R_2(\cos(2c_2r)-1)
\end{pmatrix}.
$$
Taking the derivative in $c_1, c_2$ yields
$$
\dd_{c_1,c_2}G=\begin{pmatrix}
2rR_1\cos(2c_1r) & 2rR_2\cos(2c_2r)\\
-2rR_1\sin(2c_1r) & 2rR_2\sin(2c_2r)
\end{pmatrix}.
$$
The determinant of $\dd_{c_1,c_2}G$ is equal to
$$
4r^2R_1R_2\left(\cos(2c_1r)\sin(2c_2r)-\sin(2c_1r)\cos(2c_2r)\right)=r^2\sin(2r(c_1+c_2)).
$$
From Figure \ref{fig6} we deduce that $0<2r(c_1+c_2)<\pi$ for all $0<r<1$. Thus the determinant of $\dd_{c_1,c_2}G$ does not vanish when $r\neq 0$. As a consequence of the implicit function theorem $c_1, c_2$ are smoothly depending on $r$ whenever $r\neq 0$. For the case $r\to 0$ we rewrite the equations \eqref{eq:22} in terms of
$$
\tilde G: (-\rho,\rho)\times(0,\infty)^2\to \R^2;\qquad (r, c_1,c_2)\to \begin{pmatrix}
2R_1\tau(2c_1r)c_1+2R_2\tau(2c_2r)c_2-1\\
4R_1\sigma(2c_1r)c_1^2-4R_2\sigma(2c_2r)c_2^2
\end{pmatrix},
$$
where $\tau,\sigma: \R\to\R$ are the smooth functions given by
$$
\tau(x)=\frac{\sin(x)}{x},\qquad\sigma(x)=\frac{1-\cos(x)}{x^2}.
$$
The equations \eqref{eq:22} are now equivalent to $\tilde G(r,c_1,c_2)=0$, taking the derivative in $c_1, c_2$ we obtain
$$
\dd_{c_1,c_2}\tilde G=\begin{pmatrix}
4R_1 r\tau'(2c_1r)c_1+2R_1\tau(2c_1r) & 4R_2r\tau'(2c_2r)c_2+2R_2\tau(2c_2r)\\
8R_1r\sigma'(2c_1r)c_1^2+8R_1\sigma(2c_1r)c_1 & -8R_2r\sigma'(2c_2r)c_2^2-8R_2\sigma(2c_2r)c_2
\end{pmatrix}.
$$
Observe that taking the limit $r\to 0$ in \eqref{eq:22} yields
$$
    R_1c_1(0)+R_2c_2(0)=1/2,\qquad
    R_1c_1(0)^2-R_2c_2(0)^2=0,
$$
thus
$$
2c_1(0)=(R_1+\sqrt{R_1R_2})^{-1}=\left(R_1+\frac{1}{2}\right)^{-1},\qquad
2c_2(0)=(R_2+\sqrt{R_1R_2})^{-1}=\left (R_2+\frac{1}{2}\right)^{-1}.
$$
Further, 
$\tau(0)=1=2\sigma(0),\ \tau'(0)=0=\sigma'(0)$ and therefore
$$
\dd_{c_1,c_2}\tilde G\vert_{r=0}=\begin{pmatrix}
2R_1 & 2R_2\\
4R_1c_1(0) & -4R_2c_2(0)
\end{pmatrix}.
$$
We see that 
$$
\det\left(\dd_{c_1,c_2} \tilde G\vert_{r=0}\right)=-8R_1R_2(c_1(0)+c_2(0))=-4\neq 0
$$
and it follows by the implicit function theorem that $c_1, c_2$ are smooth at $r=0$.
\end{proof}
\noindent
This finishes the construction of the smooth map $F:D_\rho\cp^n\to\cp^n\times\cp^n$. But is it a bijection? We can explicitly give an inverse. A point $(a,b)\in \cp^n\times\cp^n\setminus\bar\Delta$ defines a unique shortest geodesic $\gamma$ connecting from $a$ to $b$. This geodesic is again a segment of a round circle in an affine copy of $\R^2\subset \mathfrak{su}(n+1)$. This situation is depicted in figure \ref{f4}. As shown in figure \ref{f4} there is a unique orthogonal decomposition of the vector $R_1a-R_2b$ such that one of the vectors has length $R_1-R_2$ and $x$ lies inside the angle $2r(c_1+c_2)$. This decomposition determines $jv$ and $(R_1-R_2)x$ and thus $(x,v)\in D_1\cp^n$ and yields an inverse to $F$. So far we showed that $F$ is a smooth, equivariant bijection.
\begin{figure}
	\centering
  \includegraphics[width=0.6\textwidth]{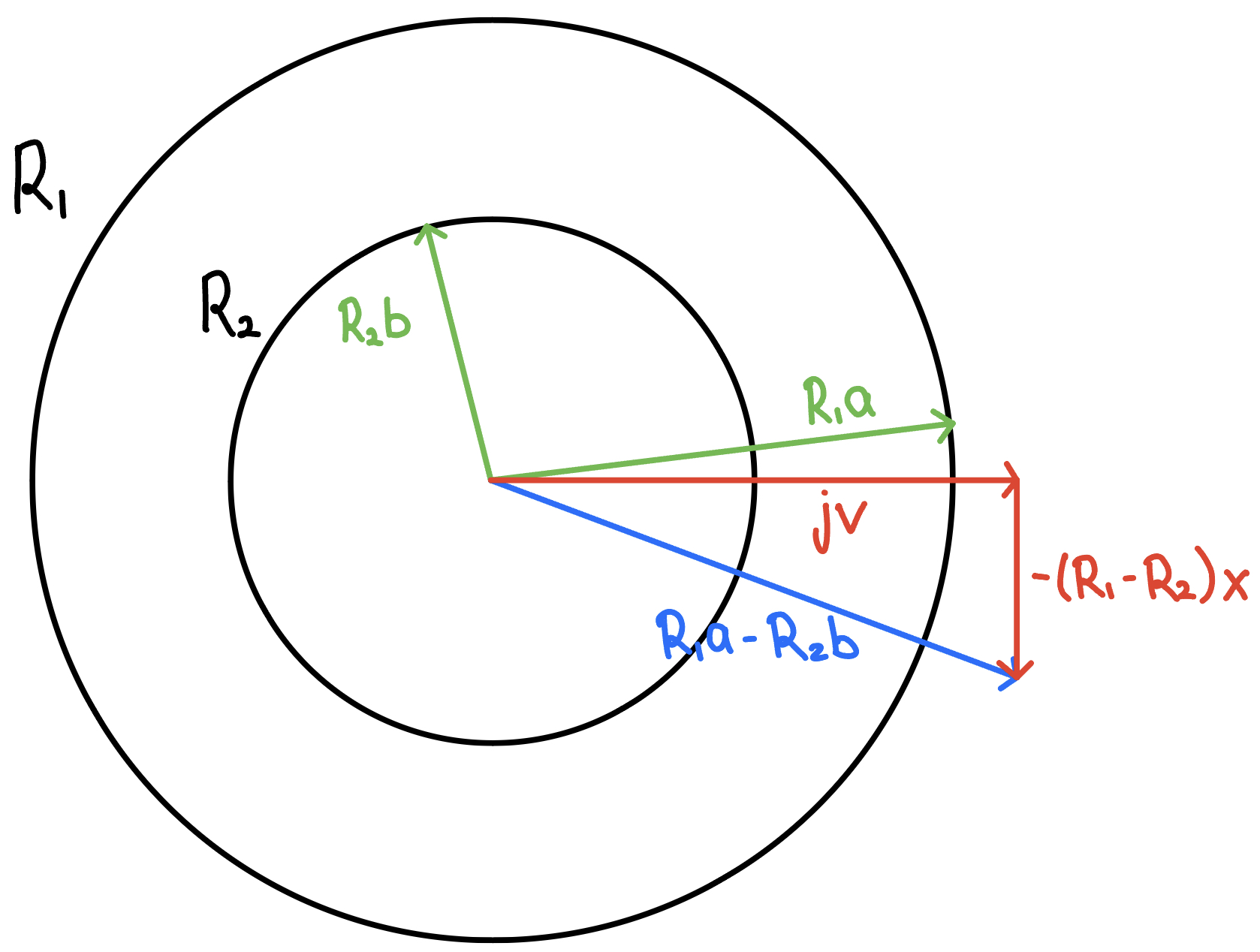}
	\caption{\textit{The vector $R_1a-R_2b$ (in blue) is given and we need to decompose it into two orthogonal vectors $jv$ and $-(R_1-R_2)x$ (in red). The second vector must be of length $R_1-R_2$, thus there are at most two possible candidates. We pick the one, such that $x$ lies inside the angle $2r(c_1+c_2)$.}}
    \label{f4}
\end{figure}
\\
\ \\
Observe that the codimension of the orbits of $SU(n+1)$ is one. Thus any complement of $\mathcal{D}$ is 1-dimensional and thus isotropic for both symplectic forms $\omega_s$ and $F^*(R_1\omega_{FS}\ominus R_2\omega_{FS})$. It follows from Lemma \ref{lem1} that $F$ is a symplectomorphism.
\end{proof}

\subsection{Hofer--Zehnder capacity}
 The symplectomorphism $F$ can also be used to compute the Hofer--Zehnder capacity of $(D_1 \cp^n,\dd\lambda-s\pi^*\omega_{FS})$.
 \begin{Theorem}[Thm.\ \ref{thmb}] Denote $\omega_{FS}$ the Fubini-Study form and $l$ the length of the geodesics, then
    $$
     c_{HZ}(D_1 \cp^n,\dd\lambda-s\pi^*\omega_{FS})=l\left(\sqrt{s^2+ 1}-\vert s\vert \right).
    $$
 \end{Theorem}
\begin{proof}
The upper bound works very similar to the untwisted case. We again pick $A\in [S^2,\cp^n\times\cp^n]$ the homotopy class of one of the two generators. As discussed for the untwisted case, the moduli space of $j\ominus j$-holomorphic spheres in class $A$ through a fixed point and intersecting the anti-diagonal is a smooth, compact $S^1$-manifold not cobordant to the empty set. It follows by Hofer--Viterbo's Theorem that $\omega(A)$ bounds the Hofer--Zehnder capacity from above. Observe that in this case it matters which factor we choose, as both yield different upper bounds. We obtain
\begin{equation}\label{e7}
    c_{HZ}(D_1\cp^n,\dd\lambda-s\pi^*\omega_{FS})\leq 4\pi \min \lbrace R_1, R_2\rbrace= 2\pi \left( \sqrt{s^2+1}-\vert s\vert\right).
\end{equation}
For the lower bound we need to work a little harder, as we can not use the same Hamiltonian as in the untwisted case. Indeed the kinetic Hamiltonian $E(x,v):=\frac{1}{2}g_x(v,v)$ does not generate geodesic flow anymore but rather its magnetic version.
\begin{Lemma}
    The magnetic geodesic flow (i.e. the flow generated by $X_E$) is totally periodic. 
\end{Lemma}
\begin{proof}
    As shown in \cite[Lemma 6.1]{BR19} the vector field $X_E$ is pointwise given by
    $$
    (X_E)_{(x,v)}=X_{(x,v)}+s(jv)^\VV,
    $$
    where $X$ denotes the vector field generating geodesic flow and $(\cdot)^\VV$ the vertical lift as introduced in the proof of Lemma \ref{lem2}. Now recall that for any point $(x,v)\in T\cp^n$ outside the zero section there is a unique totally geodesic embedded copy of $\cp^1$ through $x$ tangent to $v$. In particular the vector fields are $X+s(jv)^\VV$ are tangent to these embedded copies of $T\cp^1$ and their flow will restrict to $T\cp^1$. It is therefore enough to study the one dimensional case. So let $\gamma(t)=(x(t),v(t))$ be a magnetic geodesic, i.e. $\dot\gamma=X+s(jv)^\VV$. Projecting to the vertical distribution we obtain
    $\dot x (t)= v(t)$. Now projecting to the horizontal distribution we 
    obtain 
    $$
    \nabla_{\dot x}v=sjv.
    $$
    It follows that $\gamma(t)=(x(t),\dot x(t))$, where $x(t)$ is a curve of constant geodesic curvature $\kappa_g=\vert \frac{s}{v}\vert$, i.e. it parametrizes a geodesic circle and is therefore closed. 
\end{proof}
\noindent
 If $R$ denotes the radius (with respect to the Riemannian metric $g$) of a geodesic circle we know using normal polar coordinates that its circumference $C$ and the geodesic curvature $\kappa_g$ are
\begin{align}\label{e5}
   C&=\frac{2\pi}{\sqrt{\kappa}}\sin(\sqrt{\kappa}R)=\frac{2\pi\sqrt{\kappa}^{-1}\tan(\sqrt{\kappa}R)}{\sqrt{1+(\tan(\sqrt{\kappa}R))^2}},\\ 
   \kappa_g&=\frac{\sqrt{\kappa}}{\tan(\sqrt{\kappa}R)}.\label{e4}
\end{align} 
Inserting $\kappa_g$ into \eqref{e5} yields
$$
C=\frac{2\pi}{\kappa_g\sqrt{1+\kappa/\kappa_g^2}}=\frac{2\pi\vert v\vert}{\sqrt{s^2+\kappa\vert v\vert^2}}.
$$
Now, we conclude that the period is given by
$$
T=\frac{C}{\vert v\vert}=\frac{2\pi}{\sqrt{s^2+\kappa\vert v\vert^2}}.
$$
In particular the reparametrization
$$
H: D_1\cp^n\to \R;\ \ (x,v)\mapsto 2\pi\left( \sqrt{s^2+\vert v\vert^2}-\vert s\vert\right)
$$
is the moment map of a circle action on $(D_1 \cp^n,\dd\lambda-s\pi^*\omega_{FS})$. 
Arguing as in the untwisted case at the beginning of section \ref{sec3} we can modify $H$
changing its oscillation only by an arbitrary small $\varepsilon>0$ to make $H$ admissible. In particular we obtain the following lower bound
$$
c_{HZ}(D_1\cp^n)\geq \max H-\min H=2\pi \left( \sqrt{s^2+1}-\vert s\vert\right).
$$
Comparing this to equation \eqref{e7} we see that upper and lower bound agree, further we choose the normalization of the metric such that $l=2\pi$, therefore we just finished the proof of Theorem \ref{thmb}.
\end{proof}
\begin{Remark}
    We can interpret the symplectomorphism of Theorem \ref{thm3} as a Lerman cut with respect to $H$, i.e.
    $$
\overline{\left( D_\rho \cp^n, \omega_s\right)}\cong \left(\cp^n\times\cp^n,R_1\omega_{FS}\ominus R_2\omega_{FS}\right).
    $$
\end{Remark}
\begin{proof}
    As $H$ generates a circle action also $H\circ F^{-1}$ must generate a Hamiltonian circle action on $(\cp^n\times\cp^n\setminus \bar \Delta, R_1\omega_{FS}\ominus R_2\omega_{FS})$. Denote $(x,v):= F^{-1}(N,S)$. By construction $\mu_1(x,v)=\mu_2(N,S)$ and therefore
\begin{align*}
   \vert \mu_1(x,v)\vert ^2 = \vert \mu_2(N,S)\vert^2 \ &\Leftrightarrow\ \vert v\vert^2+ (R_2-R_1)^2
   =R_1^2+R_2^2-2R_1 R_2(N,S) \\
   &\Leftrightarrow\ \vert v\vert^2=2R_1 R_2(1-(N,S)). 
\end{align*}
Thus
$$
H(F^{-1}(N,S))=2\pi \left (\vert R_1 N-R_2 S\vert -\vert R_2-R_1\vert \right),
$$
in particular $H\circ F^{-1}$ extends smoothly to $\cp^n\times \cp^n$ and is critical on the anti-diagonal divisor $\bar\Delta$.
\end{proof}
\newpage
\bibliographystyle{abbrv}
\bibliography{ref}

\end{document}